\documentclass[12pt]{article}
\usepackage{amsmath,amsthm,amssymb,latexsym,color}
\usepackage{bbm}
\usepackage{hyperref}

\theoremstyle{plain}
\newtheorem{theor}{Theorem}[section]

\newtheorem{prop}[theor]{Proposition}

\newtheorem{lemma}[theor]{Lemma}
\newtheorem{rem}[theor]{Remark}
\theoremstyle{remark}

\def\R{{\mathbb R}}

\def\Z{{\mathbb Z}}
\def\C{\mathbb{C}}

\def\Re{\mbox{Re}}
\def\Im{\mbox{Im}}

\def\Prob{{\mathbb P}}
\def\Exp{{\mathbb E}}
\def\Event{{\mathcal E}}

\def\row{R}
\def\Proj{P}
\def\spn{{\rm span}\,}
\def\supp{{\rm supp }}

\def\Tr{{\mathrm{Tr}}}
\def\idmat{{\rm Id}}
\def\d{{\rm dist}}
\def\Net{{\mathcal N}}

\newcommand{\eps}{\varepsilon}

\def\HS{{\rm HS}}

\newcommand{\Mc}{\mathcal{M}_{n,d}}

\def\Om0{\Omega_0}

\def\HS{{\rm HS}}
\def\B{{\mathcal B_n}}
\def\A{A_n}

\title{Circular law for sparse random regular digraphs}
\author{
Alexander E. Litvak
\and
Anna Lytova
\and
Konstantin Tikhomirov
\and
Nicole Tomczak-Jaegermann
\and
Pierre Youssef
}
\newcommand\address{\noindent\leavevmode

\medskip
\noindent
Alexander E. Litvak
and Nicole Tomczak-Jaegermann,\\
Dept.~of Math.~and Stat.~Sciences,\\
University of Alberta, \\
Edmonton, AB, Canada, T6G 2G1.\\
\texttt{\small
e-mails:  aelitvak@gmail.com \, \, and \, \,
nicole.tomczak@ualberta.ca}\\

\medskip

\noindent
Anna Lytova,\\
Faculty of Math., Physics, and Comp. Science,\\
University of Opole,\\
plac Kopernika 11A, 45-040,\\
Opole, Poland.\\
\texttt{\small
e-mail:
alytova@math.uni.opole.pl
}\\

\medskip

\noindent
 Konstantin Tikhomirov,\\
Dept.~of Math.,
Princeton University,\\
Fine Hall, Washington road,\\
Princeton, NJ 08544.\\
\texttt{\small
e-mail:   kt12@math.princeton.edu}\\

\medskip

\noindent
Pierre Youssef,\\
Universit\'e Paris Diderot,\\
Laboratoire de Probabilit\'es et de mod\`eles al\'eatoires,\\
75013 Paris, France.\\
\texttt{\small
e-mail:  youssef@math.univ-paris-diderot.fr}
}

\begin{document}

\maketitle

\begin{abstract}
Fix a constant $C\geq 1$ and let $d=d(n)$ satisfy $d\leq \ln^{C} n$ for every large integer $n$.
Denote by $\A$ the adjacency matrix of a uniform random directed $d$-regular graph
on $n$ vertices. We show that, as long as $d\to\infty$ with $n$, the empirical spectral distribution of
appropriately rescaled matrix $\A$ converges weakly
in probability to the circular law. This result, together with an earlier work of Cook,
completely settles the problem of weak convergence of the empirical distribution in directed $d$-regular setting
with the degree tending to infinity.
As a crucial element of our proof,
we develop a technique of bounding intermediate singular values of $\A$ based on
studying random normals to rowspaces and on constructing a product structure to deal with the lack of independence
between the matrix entries.
\end{abstract}

\noindent
{\small \bf AMS 2010 Classification:}
{\small
primary: 60B20, 15B52, 46B06, 05C80;
secondary: 46B09, 60C05
}


\section{Introduction}
\label{s:intro}

Given an $n\times n$ random matrix $B$, its \textit{empirical spectral distribution} (ESD)
is the random probability measure on $\C$ given by
$$
\mu_B:=\frac{1}{n}\sum_{i=1}^n \delta_{\lambda_i},
$$
where $(\lambda_i)_{i\leq n}$ denote the eigenvalues of $B$
(with multiplicities counted, and enumerated in arbitrary order).
The study of the empirical spectral distribution is one of the major research directions in the theory of random matrices,
with applications to other fields \cite{Mehta, AGZ, BS, PaSh, EY}.
A fundamental fact in this area is the universality phenomenon which
asserts that under very general conditions the empirical spectral distribution
and some other characteristics of a random matrix
asymptotically behave similarly to the empirical distribution (or corresponding characteristics)
of the Gaussian random matrix of an appropriate symmetry type.
This phenomenon has been confirmed for various models and in various senses (including limiting laws for the ESD,
local eigenvalue statistics, distribution of eigenvectors).
We refer to monographs \cite{AGZ, BS, PaSh, EY} for a (partial) exposition of the results.

In case of non-Hermitian random matrices with i.i.d.\ entries, the limit of the empirical spectral distribution
is governed by the {\it circular law}.
Compared to ESD's of the {\it Wigner} (Hermitian with i.i.d.\ entries above the diagonal) and {\it Wishart} (sample covariance matrices),
the study of the spectral distribution in the non-Hermitian setting is complicated by its instability under small perturbations
of the matrix entries, and by the fact that some of the standard techniques, involving the moment method and
truncation of the matrix entries,
fail in the non-Hermitian case (we refer to \cite[Section~11.1]{BS} for more information).
As a specific example, while the bulk of the ESD of Hermitian matrices
is stable under small-rank perturbations due to interlacing properties,
the spectrum of random non-Hermitian matrices can be very sensitive even to a rank one perturbation
(see \cite[Example~11.1]{BS} or \cite[Example~1.2]{BC}).

Denote by $\mu_{circ}$ the unifom probability measure on the unit disk of the complex plane,
that is
$$\mu_{circ}=\pi^{-1}{\bf 1}_{|z|\leq 1}.$$
Convergence of the appropriately rescaled empirical spectral distribution of the standard Gaussian matrix with i.i.d.\ complex entries
was derived in the first edition of monograph \cite{Mehta} (see \cite[Chapter~15]{Mehta}), and, much later,
a corresponding result in the real case was obtained in \cite{Edelman ESD}.
Both  results relied on the explicit formula for joint distribution of eigenvalues, which is available
in the Gaussian setting
\cite{Ginibre}.
The circular law for non-Gaussian matrices with bounded densities of the entries
was verified in \cite{Bai circular};
the density condition was removed in \cite{PZ, GT, TV10}, with paper \cite{TV10} establishing the circular law
for the i.i.d.\ model under weakest moment assumptions.
The sparse i.i.d.\ model was considered in papers \cite{TV sparse, GT, BasRud circ}.
We refer to \cite{BC} for a detailed exposition and historical overview of the circular law in the i.i.d.\ setting,
and for further references.
For a review of other recent developments, including the limiting laws for inhomogeneous matrices
and the local circular law, we refer to the introduction of \cite{Cook-circ}.

\smallskip

In this paper, we are concerned with a sparse model of random matrices whose entries are not independent.
In what follows, for every positive integers $d\leq n$
we denote by
$\Mc$ the set of all $n\times n$ matrices whose entries take values in $\{0,1\}$ and
the sum of elements within each row and each column is equal to $d$.
In other words, $\Mc$ is the set of adjacency matrices of $d$-regular directed graphs on $n$ vertices,
where we allow loops but do not allow multiple edges. We consider the random matrix $\A$ uniformly distributed on $\Mc$.
Random directed $d$-regular graphs provide a basic model of a {\it typical} graph with predefined
in- and out-degree sequences
and in this connection are of interest in network analysis.
In more general setting, random (weighted) directed graphs are used to model connections between neurons and the
eigenvalue distribution of their adjacency matrices (the {\it synaptic matrices} for the neural networks)
has been given considerable attention in literature. We refer to the introduction of \cite{Cook-circ}
for a discussion of those works.

In the directed $d$-regular setting, it was conjectured (see \cite[Section 7]{BC})
that for any fixed $3\leq d\leq n-3$, $\mu_{\A}$ converges to the probability measure
$$
 \frac{1}{\pi}\frac{d^2(d-1)}{(d^2-\vert z\vert^2)^2} \mathbf{1}_{\{\vert z\vert <\sqrt{d}\}}\, dx\, dy.
$$
as $n$ goes to infinity.
This measure is usually referred to as the {\it oriented Kesten--McKay distribution},
a non-symmetric version of the classical Kesten--McKay law for the limiting ESD of random undirected $d$-degular graphs
\cite{kesten, mckay, BHY}. Up to rescaling by $\sqrt{d}$, this measure tends to the circular law as $d$ tends to infinity.
Proving the above conjecture remains a major challenge as of this writing.

\smallskip

In this paper we establish the circular law for sparse random directed $d$-regular graphs for any $d$ going to infinity with $n$.
We prove the following theorem.

\begin{theor}[The circular law]\label{t: circular law}
Fix a constant $C\geq 1$ and for any $n>1$ let $d=d(n)$ be a positive integer satisfying $d\leq \ln^{C} n$.
Assume that $d\to\infty$ with $n$.
Then the sequence of empirical spectral distributions $(\mu_{d^{-1/2}\A})_{n}$ corresponding to $\A\in\Mc$
converges weakly in probability to the uniform distribution on the unit disk of the complex plane.
\end{theor}

The circular law for $d$-regular digraphs
in the range $\ln^{96} n\leq \min(d,n-d)$ was verified in earlier work \cite{Cook-circ} (see also \cite{BCZ}).
Thus, our Theorem~\ref{t: circular law} closes the gap between known limiting distribution for denser $d$-regular digraphs
and the conjectured oriented Kesten--McKay limiting distribution for $d$-regular digraphs of constant degree.
The proof of Theorem~\ref{t: circular law} combines some known methods
used previously in works on the circular law,
with crucial new ingredients related to estimating the {\it intermediate} singular values of the shifted adjacency matrix.
The rest of the introduction is divided into two parts. In the first part, we recall known techniques
(such as Hermitization)  and previously established facts
about $d$-regular digraphs that will be needed for the proof. In the second part, we discuss limitations of existing tools
(see remarks after Proposition~\ref{prop2: cook-anti}) and consider our approach to bounding intermediate singular values of $\A-z\,\idmat$.

\smallskip

As in works \cite{Girko, Bai circular, GT, TV10} dealing with the i.i.d.\ setting, a key element in the proof of the circular law
for $d$-regular digraphs is
to transport the problem of the limiting ESD
to the singular values distribution, which is much easier to study.
This method -- called the Hermitization technique --
goes back to Girko \cite{Girko} and exploits
a close relation between the log-potential functions of the spectral and singular values distributions.
Following Girko, this idea was used in various papers dealing with non-Hermitian random matrices, in particular \cite{Bai circular, GT, TV10}.
The Hermitization technique is presented in literature in somewhat different forms;
below we follow the exposition in \cite{BC}.

The \textit{singular values distribution} of an $n\times n$ random matrix $B$ is the random probability measure on $\R$ given by
$$
\nu_B:=\frac{1}{n}\sum_{i=1}^n \delta_{s_i},
$$
where $(s_i)_{i\leq n}$ denote the singular values of $B$.
Everywhere in this paper, we use non-increasing ordering for the singular values,
so that $s_1=s_1(B)$ is the largest one and $s_n=s_n(B)$ is the smallest one.

The \textit{logarithmic potential} $U_\mu: \C\to (-\infty,\infty]$ of a probability measure $\mu$ on $\C$
is defined for any $z\in \C$  by
$$
U_\mu(z):=-\int_{\C} \ln\vert z-\lambda\vert \,d\mu(\lambda).
$$
The logarithmic potential function uniquely determines the underlying measure,
that is, if $U_{\mu}=U_{\mu'}$  Lebesgue almost everywhere then $\mu=\mu'$ (see, in particular, \cite[Lemma~4.1]{BC}).

Given an $n\times n$ matrix $B$, it is easy to check that
$$
U_{\mu_B}(z)= -\frac{1}{n} \ln\vert {\rm det}(B-z\idmat)\vert= -\int_{0}^\infty \ln(t) \,d\nu_{B-z\idmat}(t)
=-\frac{1}{n}\sum\limits_{i=1}^n \ln(s_i(B-z\idmat)).
$$
Therefore, knowing $\nu_{B-z\idmat}$
for almost all $z\in \C$, we can determine $U_{\mu_B}$, hence $\mu_B$ itself.
This observation lies at the heart of the method.
Below we state its formalized version.

\begin{lemma}[{Hermitization, see \cite[Lemma~4.3]{BC}}]\label{lem: hermitization}
For each $n$, let $B_n$ be an $n\times n$ complex random matrix,
and assume that for Lebesgue almost all $z\in \C$, one has
\begin{itemize}
\item[(i)] There exists a probability measure $\nu_z$ on $\R_+$ such that $\nu_{B_n-z\idmat}$ tends weakly to $\nu_z$ in probability;
\item[(ii)] $\ln$ is uniformly integrable for $\nu_{B_n-z\idmat}$ in probability, i.e.\ for every $\varepsilon>0$ there exists
$T=T(z,\varepsilon)<\infty$ such that
$$
\sup_{n} \Prob\left\{ \int_{\{\vert \ln(s)\vert >T\}} \vert \ln(s)\vert\, d\nu_{B_n-z\idmat}(s)>\varepsilon\right\}\leq \varepsilon.
$$
\end{itemize}
Then $\mu_{B_n}$ converges weakly in probability to the unique probability measure $\mu$ on
$\C$ whose logarithmic potential function is given by
\begin{equation}\label{eq: hermit-log}
U_\mu(z)=-\int_{0}^\infty \ln(s)d\nu_z(s).
\end{equation}

\end{lemma}

Thus in order to establish the circular law, one needs to show the convergence of
the empirical singular values distribution and the uniform integrability of the logarithm.
For the first part, we will rely on a recent result of Cook \cite{Cook-circ},
who uses the above strategy in order to establish
the circular law for the uniform model on $\Mc$ for $d\geq \ln^{96}n$.
The following is a version of Proposition~7.2 in \cite{Cook-circ}. Note that its
proof doesn't require that $d$ is at least polylogarithmic in $n$; just $d\to \infty$ is enough.

\begin{prop}[{Weak convergence of singular values distributions, \cite{Cook-circ}}]
\label{p:prop-conv-esv}
Assume that $d=d(n)=o(\sqrt{n})$ and $d\to\infty$ together with $n$.
Then for each $z\in \C$, there exists a probability measure $\nu_z$ on $\R_+$
such that $\nu_{d^{-1/2} A_n- z\idmat}$ converges weakly in probability to $\nu_z$ as $n\to\infty$.
Moreover, the family $\{\nu_z\}_{z\in \C}$ satisfies \eqref{eq: hermit-log} with $\mu=\mu_{circ}$.
\end{prop}

In fact in \cite{Cook-circ}, the above proposition was stated for the centralized matrix
$$
  X_n=A_n-\frac{d}{n} \textbf{1}\textbf{1}^t
$$
instead of $A_n$. However,
since these two matrices differ by a rank one matrix,
then using the interlacing of their singular values one can deduce that their empirical singular value distributions satisfy
$$
 \sup_{a>0} \Big\vert \nu_{d^{-1/2} A_n- z\idmat} ([0,a])
 -\nu_{d^{-1/2} X_n- z\idmat} ([0,a])\Big\vert \leq \frac{1}{n}
$$
(this has been also used in \cite{Cook-circ}, see formula (7.6) there).
Therefore the two corresponding singular values distributions exhibit the same limiting behavior.

From the above, it is clear that the main obstacle in
establishing Theorem~\ref{t: circular law} is in showing the uniform integrability of the logarithm.
More precisely, for any $\eps\in (0,1)$ and any $z\in\C$
one needs to show that there is $T=T(z,\varepsilon)>0$ such that with probability going to one as $n\to\infty$,
\begin{equation}\label{eq: sum-log}
\sum_{i:\,|\ln s_i(B_z)|\geq T}|\ln s_i(B_z)|\leq \varepsilon n,
\end{equation}
where we set $B_z:=d^{-1/2} \A- z\idmat$. A simple computation involving the Hilbert--Schmidt norm of
$B_z$ shows that
the main contributors to the above sum are small singular values, i.e.\ those smaller than $e^{-T}$.

Further, building upon ideas in \cite{Cook-inv} as well as the authors' works \cite{LLTTY:15, LLTTY-cras},
in \cite{LLTTY-sing} a polynomial lower bound on the smallest singular value of $B_z$ was obtained.

\begin{theor}[{\cite{LLTTY-sing}}]\label{t: smin}
There exists a universal constant $C\geq 1$ such that for all positive integers $d, n$ satisfying
$C\leq d\leq  n/\ln^2 n$ and every $z\in \C$ with $|z|\leq d/6$ one has
$$\Prob\Big\{s_{\min}(\A-z\,\idmat)\geq n^{-6}\Big\}\geq 1- d^{-1/4}.$$
\end{theor}
The above came as an improvement (in the sparse regime) of an earlier estimate of Cook \cite{Cook-circ},
who derived his result under an additional assumption $d\geq \ln^{C}n$ for a universal constant $C$.
Theorem~\ref{t: smin} immediately shows that the contribution of $o(n/\ln n)$ least singular values
to the sum in \eqref{eq: sum-log} is negligible.

Together with the observation concerning largest singular values, this leaves the
task of estimating the sum
\begin{equation}\label{eq: aux sum processed}
\sum_{\stackrel{i\leq n-o(n/\ln n):}{s_i(B_z)\leq e^{-T} }}|\ln s_i(B_z)|.
\end{equation}
Partially, the estimate comes from the following result of \cite{Cook-circ} obtained via comparison with
Bernoulli random matrices.

\begin{prop}[{\cite[Proposition~7.3]{Cook-circ}}]\label{prop2: cook-anti}
There are absolute constants $C> 1>c>0$ such that the following holds.
Let $C\leq d\leq  n$ be positive integers and $z\in \C$. Assume that $d=d(n)=o(\sqrt{n})$
and $d\to\infty$ together with $n$.
Then for all large $n$ with probability at least $1-\exp(-n/2)$, one has for every $k\leq  n-C nd^{-1/48}$,
$$
s_k(B_z)\geq c\, \frac{n-k}{n}.
$$
\end{prop}

This proposition is stated in \cite{Cook-circ} for $d$ polylogarithmic in $n$.
In Section~\ref{s: intermediate} (see Remark~\ref{rem: cook-smaller d}), we indicate the changes to be
made in \cite{Cook-circ} to derive Proposition~\ref{prop2: cook-anti} without this restriction on $d$
(the change is actually implicitly mentioned in \cite{Cook-circ}).

Proposition~\ref{prop2: cook-anti} can be viewed as a (weak local) form of the {\it Marchenko--Pastur law} for the singular values distribution
\cite{MP, Yin}.
When $d$ is at least polylogarithmic in $n$ (with an appropriate power of the log)
the proposition is enough to cover the whole range of singular values in \eqref{eq: aux sum processed} and complete the proof.
This is the approach realized in \cite{Cook-circ}.
However, when $d$ is smaller the power of $\ln n$, the above result
leaves untreated the range of smallish singular values from $s_{n-C nd^{-1/48}}$ to $s_{n-o(n/\ln n)}$.

The idea of the proof of  Proposition~7.3 in \cite{Cook-circ} is to compare the uniform
directed $d$-regular model
with the directed Erd\H{o}s--Renyi graph, that is, to replace the matrix $\A$ by a
matrix $\B$ with i.i.d.\ Bernoulli random variables with the parameter $d/n$.
At this step, one has to condition on the event that the Erd\H{o}s--Renyi graph is $d$-regular,
which is of very small probability superexponential in $n$ \cite{MW03}.
This way, satisfactory estimates for the intermediate singular values of the shifted adjacency matrix $\A-z\,\idmat$
can be obtained only if very strong estimates are available in the Bernoulli setting, which hold with probability at least $1-\exp(-\omega(n))$.
Currently, no estimates of this type are available in the very sparse regime, moreover, it is not clear whether
such strong estimates can be obtained at all.
This forces us to develop a completely different approach to bound the singular values $s_k$ of $\A-z\,\idmat$ in the range
${n-C nd^{-c}} \leq k \leq {n-o(n/\ln n)}$.

\begin{theor}[Intermediate singular values]\label{th: inter-sv}
There exists a universal constant $C\geq 1$ with the following property.
Let $d$, $n$ be integers satisfying  $C\leq d\leq \ln^{96} n$ and let
$z\in\C$ be such that $\vert z\vert \leq \sqrt{d}\ln d$.
Then for all
$$
 n-  2nd^{-3/2}\leq k\leq n- 3n/\ln^{144} n
$$
one has
$$
\Prob\Big\{ \A\in\Mc:\,s_{k}(\A-z\,\idmat)\geq \exp\Big(-C\Big(\frac{n}{n-k}\Big)^{1/144}\Big)\Big\}\geq 1-C \, \frac{n-k}{ n}.
$$
In particular,
$$
\Prob\Big\{\A\in\Mc:\,s_{k}(\A-z\,\idmat)\geq \exp\big(-C\, d^{1/96}\big) \text{ for all }  k\leq n- 2nd^{-3/2}\Big\}\geq 1- \frac{C}{d^{3/2}}.
$$
\end{theor}

In the above, we restricted our analysis to $d\leq \ln^{96} n$ as it complements what is covered by
Proposition~\ref{prop2: cook-anti}. Our approach can be extended to higher powers  of $\ln n$
(even possibly for any $d\leq \exp(\sqrt{\ln n})$ as in \cite{LLTTY-ker}), however we prefer to prove
the above formulation as it is sufficient for our purposes and improves the exposition.
Equipped with Theorem~\ref{t: smin}, Proposition ~\ref{prop2: cook-anti}, and Theorem~\ref{th: inter-sv}, we have bounds on all singular values which would allow us to show the uniform integrability of the logarithm and thus to establish the circular law. We note that the idea of splitting the singular values into different regimes is standard in this context (see \cite[Chapter 2, Section 8]{tao-book} for more details) as one needs  different levels  of precision depending on the magnitude of the singular values. In our case, the sparsity adds a serious challenge and the comparison methods described previously are ineffective.
Moreover, due to the lack of independence, standard approaches to estimating the singular values are not applicable in our setting.
For example, one cannot use Talagrand's concentration inequality \cite[Theorem~2.1.13]{tao-book} in this context the same way as was
previously done in the literature (see, in particular, \cite{TV10}).
The issues appear when following the standard scheme which reduces estimates for the singular values to distance estimates for the matrix rows. Namely, the second moment identity \cite{tao-book} or the restricted invertibility principle
(see, for example, \cite[Theorem~9]{NY}) relates the intermediate singular values to quantities of the form
$$
 \d\big(\row_i(B_z), \spn\{\row_j(B_z)\}_{j\in I}\big),
$$
for $I\subset [n]$ and $i\in [n]\setminus I$, where $\row_i(B_z)$ denote the $i$-th row of $B_z$.
When these rows are independent, one can condition on a realization of $E:=\spn\{\row_j(B_z)\}_{j\in I}$
then use the randomness of the $i$-th row together with standard anti-concentration arguments to get a lower bound for
$
 \Vert\Proj_{E^\perp} \row_i(B_z)\Vert_2= \d(\row_i(B_z), E).
$
 On the other hand,
the randomness of $E$ is used to ensure that its normal vector is well spread for the
anti-concentration argument to work.
In our setting, i.e. for random $d$-regular graphs, the lack of product structure adds serious
complications to the problem.
Studying the distribution of a row conditioned on the realization of other rows involves
careful application of the
expansion properties of the underlying graph. In particular, such a direction was pursued by the
third and last named authors
\cite{TY} to establish, for denser $d$-regular graphs, a large deviation inequality for the inner
product of a row with an
arbitrary vector, conditioned on a realization of a block of rows. At the same time, the technical
approach of \cite{TY} is not applicable here
as we deal with very spars random graphs and are interested in a small ball inequality instead of
large deviations.

The key idea behind the argument developed in this paper is to inject additional randomness and create
a sort of product structure, which would allow us to use the randomness of each of the (dependent) quantities
involved. We provide a rough illustration of this idea. Fix $I\subset [n]$ and $i\in [n]\setminus I$,
and observe that
\begin{equation}\label{eq: dist-proj-gaus}
\d(\row_i(B_z), E)^2= \Vert\Proj_{E^\perp} \row_i(B_z)\Vert_2^2= \Exp_{G}\, |\langle \Proj_{E^\perp} G,  \row_i(B_z)\rangle|^2,
\end{equation}
where $G$ is a standard Gaussian vector in $\C^n$ and the expectation is taken with respect to $G$.
Now standard Gaussian concentration allows us to remove the expectation above and benefit from the randomness of $G$ to study the quantity
$ \langle \Proj_{E^\perp} G,  \row_i(B_z)\rangle$.
The vector $\Proj_{E^\perp} G$ plays the role of a uniform random normal to $E$. As the key technical ingredient,
we prove that the random normal is typically unstructured,
i.e.\ has many levels of coordinates.
In this sense, one of the most important  inputs of this paper is a statement about the kernel of submatrices
of $\A-z\,\idmat$ formed by removing a small proportion of rows (see Theorem~\ref{th: random normal structure}).
 Once equipped with this statement, we  switch back to the randomness
of $\row_i(B_z)$
in order to establish an anti-concentration inequality. Note that this also requires additional efforts as we deal with a
sum of dependent random variables with non-trivial conditional distributions
(conditioned on a realization of $E$) as opposed to the standard estimates in the independent case. The structure of normal vectors
to subspaces spanned by the rows of random $d$-regular graphs was investigated by the authors in \cite{LLTTY-ker}.
In particular, it was shown that if the subspace $E$ is of large dimension, then any normal vector to it is either very steep (has a sudden drop at the beginning of its non-increasing rearrangement)
or has a moderate coordinates decay and is unstructured (i.e.\ has many levels of coordinates).
The latter property is essential for the anti-concentration argument to be effective.
Informally speaking, one of the advantages of introducing the additional randomness lies in the fact that the random Gaussian vector picks the best normal vector and benefits from better structural properties. This vague observation will become more rigorous and clear from the proof of Theorem~\ref{th: random normal structure}. We expect that some elements of our proof can be fruitful in the study of other matrix models which lack independence.

The paper is organized as follows. In Section~\ref{s: circ}, we derive the circular law assuming the estimates
on the intermediate singular values. In Section~\ref{notat}, we introduce notations. 
In Section~\ref{s: random normal}, we prove the structural theorem
(Theorem~\ref{th: random normal structure}) for {\it uniform random normals} after providing estimates for order
statistics of projection of Gaussian vectors.  In Section~\ref{s: intermediate}, we establish an
anti-concentration estimate and combine it with the structural theorem in order to prove Theorem~\ref{th: inter-sv}.

\section{Proof of Theorem~\ref{t: circular law}}\label{s: circ}

In this section we prove Theorem~\ref{t: circular law} --- the circular law for the limiting spectral distribution ---
assuming the results mentioned in the introduction.
As discussed before, we only need to verify uniform integrability of the logarithm, that is, item~(ii) of Lemma~\ref{lem: hermitization}.

\smallskip

Fix $z\in\C$, $\varepsilon>0$
and,  given $n$ and $d$ satisfying assumptions of the theorem, set $B_z:=d^{-1/2}\A-z\idmat$.
We want to show that there exists $T=T(z,\varepsilon)>0$ such that
$$\Prob\Big\{\sum_{i:\,|\ln s_i(B_z)|\geq T}|\ln s_i(B_z)|\geq \varepsilon n\Big\}\leq \varepsilon.$$
In the proof below summation over an empty set is always assumed to give $0$.

For large singular numbers we will apply a deterministic bound which follows from
$d$-regularity, namely we will use that $\|\A\|_{\HS}^2= nd$, where
$\Vert\cdot\Vert_{\HS}$ denotes the Hilbert--Schmidt norm.
Choose a sufficiently large $T=T(z,\varepsilon)>0$ to ensure that
$$\ln x \leq \frac{\varepsilon}{4(1+\vert z\vert^2)} x^2$$ whenever $x \geq e^T$. Then
\begin{align*}
\sum_{i:\,s_i(B_z)\geq e^T}\ln s_i(B_z)& \leq \frac{\varepsilon}{4(1+\vert z\vert^2)} \sum_{i:\,s_i(B_z)\geq e^T} s_i^2(B_z) \leq  \frac{\varepsilon}{4(1+\vert z\vert^2)}\Vert B_z\Vert_{\HS}^2\\
&\leq \frac{\varepsilon}{2(1+\vert z\vert^2)} \Big(\Vert d^{-1/2}\A\Vert_{\HS}^2 + \Vert z\idmat\Vert_{\HS}^2\Big)= \frac{\varepsilon}{2} n.
\end{align*}
Note that
one could also use
the spectral gap estimate for $d$-regular graphs (see \cite{TY} and references therein),
which implies that with large probability all singular values of $d^{-1/2}\A$ except for $s_1$ are bounded above by a universal constant.

Thus it is enough to show a bound for small singular values, more precisely, it is enough to show that
$$
  \Prob\Big\{\sum_{i\in I}
 |\ln s_i(B_z)|\geq \varepsilon n/2\Big\}\leq \varepsilon,
$$
where
$$I=\{i\, :\, s_i(B_z)\leq e^{-T}\}.$$
We split the set $I$ into four parts:
\begin{align*}
&I_1 := I \cap \{i:\,i\leq n-Cn d^{-1/48}\},\quad\quad I_2:= (I \cap \{i:\;i\leq n-2 n/d^{3/2}\}) \setminus I_1,
\\
   &I_3 :=  (I \cap \{i:\;i\leq n-n/\ln^{2}n\}) \setminus (I_1 \cup I_2),
 \quad \mbox{ and }\quad I_4:=I\cap  \{ i\,\, :\, \,  i> n- n/\ln^2 n \} ,
\end{align*}
where $C\geq 1$ is the absolute constant from Proposition~\ref{prop2: cook-anti}.
Proposition~\ref{prop2: cook-anti} implies that with probability at least $1-\exp(-n/2)$,
for all $i\leq n-C d^{-1/48}n$ we have
$$
  s_i(B_z)\geq \frac{c(n-i)}{n}
$$
for an absolute constant $c\in (0,1)$.
Note that if $i\in I$ then this inequality implies
$i\geq n(1-1/(ce^{T}))$. Thus $I_1\ne \emptyset$ if and only if
$d^{1/48}\geq c e^T$, in which case $n\gg  c e^T$. Denoting
$$
  I_1':=\big\{i:\, \, n(1-1/{ce^{T}})\leq i\leq (1-C d^{-1/48})n\big\}
$$
and assuming
$I_1 \ne \emptyset$ we obtain
$$
\sum_{i\in I_1}
|\ln s_i(B_z)|
\leq \sum_{i\in I_1'}
\ln \frac{n}{c(n-i)} \leq
\sum_{k=1}^{n/c e^T} \ln \frac{n}{ck} \leq 2\int _1^{n/c e^T} \ln \frac{n}{ct} \, dt \leq  \frac{2n(T+1)}{ce^T}.
$$
For large enough $T$ and for $n \geq 2\ln (4/\eps)$, this implies
$$
  \Prob\Big\{\sum_{i\in I_1}
  |\ln s_i(B_z)|\geq \varepsilon n/8\Big\}\leq \varepsilon/4.
$$

Further, by Theorem~\ref{th: inter-sv} we obtain that for some universal constants $C',C_0$ with probability at least $1-C' d^{-3/2}$
we have
$$
  \sum_{i\in I_2}
 |\ln s_i(B_z)|
 \leq |I_2| \,  (C'd^{1/96})\leq C_0  d^{-1/96}n \leq  \varepsilon n/8,
$$
provided that $d\geq (8C_0 /\eps)^{96}$ and that $d \ln^2 d \geq |z|^2$.

Next, by Theorem~\ref{t: smin}, applied to the matrix $\A-z \sqrt{d}\, \idmat$,  with probability at least $1-d^{-1/4}$ we have
$s_{n}(B_z)\geq n^{-6}/\sqrt{d}$ and thus
$$
     \sum_{i\in I_4}  |\ln s_i(B_z)|\leq
  \sum_{i> n-n/\ln^2 n} |\ln s_i(B_z)|\leq \frac{n}{\ln^2 n} \, |\ln s_n(B_z)|\leq \varepsilon n/8,
$$
provided that $d\geq 36|z|^2$ and $7/\ln n\leq \varepsilon/8$.

 It remains to estimate the sum over $I_3$. Note that $I_3\ne \emptyset$ only if $2n/d^{3/2}\geq n/\ln^2 n$.
%
Consider  a sequence of indices $i_0,i_1,\dots$ defined by
$$
  i_u:=\lfloor n- 2^{-u}d^{-3/2}n\rfloor
$$
for $u\geq 0$ and let $u_0$ be the smallest integer such that $i_{u_0}\geq  n-n/\ln^2 n$.
Then
\begin{equation}\label{eq: reduction-range-sum}
\sum_{i\in I_3}
|\ln  s_i(B_z)|
\leq \sum_{u=0}^{ u_0-1} (i_{u+1}-i_u) |\ln  s_{i_{u+1}}(B_z)|
\leq 4 d^{-3/2} n \sum_{u=0}^{ u_0-1} 2^{-(u+1)} |\ln s_{i_{u+1}}(B_z)|.
\end{equation}
Assuming that  $d\ln^2 d \geq |z|^2$ and
applying Theorem~\ref{th: inter-sv} again we obtain that for every $0\leq u\leq u_0-1$,
$$
\Prob\Big\{s_{i_{u+1}}(B_z)\geq \exp\big(-C'\, d^{1/96} 2^{(u+1)/144}\big)\Big\}\geq 1-\frac{C'}{d^{3/2}2^{u+1}},
$$
where $C'>0$ is a universal constant.
Taking the union bound, we get with probability at least
$1-C' d^{-3/2}$,
$$|\ln s_{i_{u+1}}(B_z)|\leq C'\, d^{1/96} 2^{(u+1)/144}\quad\quad\mbox{for all  $0\leq u\leq u_0-1$.}$$
By \eqref{eq: reduction-range-sum} we obtain that with the same probability
$$
 \sum _{i\in I_3}|\ln  s_i(B_z)| \leq
 2C' d^{-3/2}n\sum_{u=0}^{u_0-1}2^{-(u+1)143/144} d^{1/96}\leq \varepsilon n/8,
$$
provided that $d\gg 1/\eps$.
Combining estimates for sums over $I_1,\dots,I_4$ we obtain the
result, provided that $d\geq d_0:=\max\{36 |z|^2,  C_2/\eps^{96}\}$ for a large universal constant $C_2>0$.

\smallskip

Finally, we would like to comment on a purely technical aspect -- why we can assume that $d\geq d_0$.
Given $n\geq 1$, let $X_n\subset \C$ be the set of all eigenvalues of all $d$-regular $n\times n$ matrices divided by
$\sqrt{d}$ (taken for all  $d\leq \ln^{96}n$).
Since  $X:=\bigcup_n X_n$ has zero Lebesgue measure it is enough to consider $z\not\in X$. Now  given
a sequence $d(n)\to \infty$, $z\in \C\setminus X$, and $\eps>0$ choose $n_0= n_0(z, \eps)$ so that
$d(n)\geq d_0$ whenever $n\geq n_0$.
Set
$$
 \rho=\rho(z, \eps):=\mbox{dist} (z, \bigcup_{n\leq n_0} X_n).
$$
Then $\rho>0$ and for every
$d$-regular $n\times n$ matrix
$A_n$ with $n\leq n_0$ the matrix $B_z$ is invertible and the norm of its inverse can be estimated in terms
of $n$, $d$, and $\rho$ (e.g., via formula for the inverse matrix, its Hilbert--Schmidt norm,
and Hadamard's inequality). Since $n\leq n_0$ and $s_{n}(B_z) =1/ \|B_z^{-1}\|$, we obtain
a lower bound on $s_{n}(B_z)$  in terms of $n_0$ and $\rho$. Therefore, taking sufficiently
large $T=T(z, \eps)$,  we get that for any $n\leq n_0$ the set $\{i:\;|\ln s_i(B_z)|\geq T\}$ is empty.


\section{Notation}\label{notat}

Given two positive integers $k\leq \ell$, we denote $[k]=\{1, ..., k\}$ and
$[k, \ell] = \{k, k+1, ..., \ell\}$. Given a sequence $(x_i)_{i=1}^n$, we denote  by
$(x_i^*)_{i=1}^n$  the non-increasing rearrangement of $(|x_i|)_{i=1}^n$.
In particular, for a given (random) vector $X$ in $\C^n$, the sequence $(X_i^*)_{i=1}^n$
is the non-increasing rearrangement of the absolute values of coordinates of $X$.
The vectors of the canonical basis of $\C^n$ are denoted by $e_1, e_2, ..., e_n$.
Given $E\subset \C^n$, the orthogonal projection on $E$ is denoted by $P_E$.
Given $J\subset [n]$, we denote by $P_J$ the orthogonal projection on the space
spanned by $e_j$, $j\in J$.  Given an $n\times n$ matrix $A$ we denote its rows
by $R_i(A)$, $i\leq n$. A set (or a subset of a certain set) of cardinality
$k$ is called $k$-set (resp., $k$-subset).

As mentioned in the introduction, for every positive integer $d\leq n$,
we denote by
$\Mc$ the set of all $n\times n$ matrices whose entries take values in $\{0,1\}$ and
the sum of elements within each row and each column is equal to $d$. In other words,  $\Mc$
is the set of adjacency matrices of directed $d$--regular graphs on $n$ vertices.
The random matrix uniformly distributed on $\Mc$ is denoted by $\A$ and
as before, we denote $B_z:=d^{-1/2}\A-z\idmat$, where $z\in \C$ and $\idmat$ is the identity matrix.
Below we often deal with a random subspace of $\C^n$ spanned by some rows of a random matrix.
Given $I\subset [n]$, we denote by $E(A_n, I)$ (resp., $E(B_z, I)$) the random subspace spanned
by the rows of $\A$ (resp., $B_z$) indexed by $I$.

The standard Gaussian variable in $\C$ is the variable $g=\xi_1 + i \xi_2$, where
$\xi_1$ and $\xi_2$ are independent real Gaussians distributed according to $\mathcal{N} (0, 1/2)$.
The standard Gaussian vector in $\C^n$ is the vector $(g_1, g_2, ..., g_n)$, where the $g_i$'s
are independent standard complex Gaussian variables. We denote this vector by $G$ and always assume that
it is independent of $A_n$.  We  use that the distribution of $G$, denoted below by $\gamma_n$, is
invariant under orthogonal transformations and that  for every orthogonal projection $P$
of rank $k\leq n$ the vector $PG$ is distributed as the standard Gaussian vector in $\mathbb{C}^k$.
In particular, for every non-degenerate subspace $E$ of $\C^n$ and every fixed
$x\in \C^n\setminus\{0\}$ one has  for every $t>0$,
\begin{equation}\label{l: line projection}
\Prob\big\{|\langle x, \Proj_{E} G \rangle|  \leq t\|\Proj_{E}x\|_2\big\} =
\Prob\bigg\{\bigg|\bigg\langle \frac{\Proj_{E} x}{\|\Proj_{E}x\|_2},  G \bigg\rangle\bigg|  \leq t\bigg\} =
\Prob\big\{|g|\leq t \big\} =  1- \exp(-t^2).
\end{equation}

In the next section we deal with {\it uniform random normals} which we define in the following way.
Let $E\subset\C^n$ be a linear subspace and $E^\perp$ denote its orthogonal complement. The
uniform random normal to $E$ is the standard Gaussian vector in the orthogonal complement of $E$.
Note that  the uniform random normal to $E$ is distributed as
$\Proj_{E^\perp}(G)$ which will often be denoted by $Y$.

\section{Uniform random normals}\label{s: random normal}

The result of this section is based on the structural theorem proved in \cite{LLTTY-ker}  (Theorem 1.1 there).
We state a special case of this theorem, in which we fix several parameters and restrict the range
of $d$ and of the index subset $\vert I^c\vert$ according to our needs.

\begin{theor}
\label{t: kernel restated}
Let $d, n$ be sufficiently large integers satisfying $d\leq \ln^{96}n$ and $z\in \C$ be such that
$\vert z\vert \leq \sqrt{d}\ln d$.
Let $a \in (d^{-1/2},1)$, $\gamma:=1/288$, and fix a subset $I\subset[n]$ satisfying
$$
 n/\ln^{\gamma^{-1}} n \leq \vert I^c\vert\leq n/d^3.
$$
Let $E=E(B_z, I)$ be the random subspace spanned
by the rows of  $B_z$ indexed by $I$.
Then with probability at least $1-1/n$
any non-zero vector $x\in E^{\perp}$
satisfies one of the two conditions:
\begin{itemize}
\item (Sloping with many levels) For all $i\leq a\vert I^c\vert$ one has $x_i^*\leq 0.9\, \frac{n^3}{i^3}
x_{a \vert I^c\vert}^*$ and for all $\lambda\in\C$,
$$
 \Big|\Big\{i\leq n:\,|x_i-\lambda |\leq \exp\big(- 2\big(n/\vert I^c\vert\big)^\gamma\big) x_{a\vert I^c\vert}^* \Big\}\Big|\leq \Big(\frac{\vert I^c\vert}{n}\Big)^{\gamma/2} n.
$$
\item (Very steep) There exists   $i\leq a\vert I^c\vert$ such that
$x_i^* > 0.9\, (n/i)^3 x_{a \vert I^c\vert}^*$ for some  $i\leq a\vert I^c\vert$.
\end{itemize}
\end{theor}

The idea, developed in this section, is that a normal vector picked uniformly at random in $E^\perp$ has
better structural properties (in fact, more ``unstructured'').
At the intuitive level, in the case of  large co-dimensional $E\subset \C^n$,
the vector $\Proj_{E^\perp}(G)$ should be typically unstructured, i.e., should not have many coordinates
of almost the same value. We will make this notion precise, by combining Theorem~\ref{t: kernel restated}
with some probabilistic arguments. The main result of this section is the following theorem.

\begin{theor}\label{th: random normal structure}
Let $d, n$ be sufficiently large integers satisfying $d\leq \ln^{96}n$ and $z\in \C$ be such that
$\vert z\vert \leq \sqrt{d}\ln d$.
Let $\gamma= 1/288$ and fix a subset $I\subset[n]$ satisfying
$$
   n/\ln^{1/\gamma} n \leq \vert I^c\vert\leq n/d^3.
$$
Let $E=E(B_z, I)$ be the random subspace spanned
by the rows of  $B_z$ indexed by $I$.
Then
\begin{align*}
\Prob\Big\{&\mbox{for every
$\widetilde J\subset[n]$ with $|\widetilde J|\leq 2\big(\vert I^c\vert/n\big)^{\gamma/2} n$ there is $\lambda\in\C$ such that}\\
&\big|\big\{j\in [n]\setminus \widetilde J:\, |\langle \Proj_{E^\perp}(G),e_j\rangle -\lambda|\leq
\exp\big(- C
\big(n/\vert I^c\vert\big)^\gamma\big) \big\}\big|> |I^c|\Big\}
\leq {|I^c|}/{n},
\end{align*}
where we take the product probability measure on $\Mc\times (\C^n,\gamma_n)$, i.e.\ assume that
$G$ and $\A$ are independent, and
 $C$ is a universal positive constant.
\end{theor}


We would like to  note that using a better version of the structural theorem, namely Theorem~\ref{t: kernel restated} of \cite{LLTTY-ker}, one could prove a more general statement covering a wider range of $d$ and $\vert I^c\vert$. Since the above statement is sufficient for our purposes, we prefer to avoid additional technicalities.

\smallskip

Theorem~\ref{t: kernel restated} states that any normal vector to $E$ which is not very steep (in the above sense) necessarily has at least $\big(n/\vert I^c\vert\big)^{\gamma/2}$ levels of coordinates. Theorem~\ref{th: random normal structure}
improves this by asserting that the uniform normal has as many as $n/\vert I^c\vert$ levels of coordinates. Also, as was noticed in \eqref{eq: dist-proj-gaus}, there is a straightforward connection  between the distance of a vector $x$ to $E$ and the inner  product of $x$ with $\Proj_{E^\perp}(G)$. This connection together with Theorem~\ref{th: random normal structure} and  anti-concentration machinery developed in Section~\ref{s: intermediate} allows to get bounds on the intermediate singular values.

\bigskip

\subsection{Order statistics of uniform random normals}

Recall that  for a given $E\subset\C^n$, $Y = Y(E) =(Y_1,\dots,Y_n)=\Proj_{E^\perp}(G).$
We also deal with linear combinations of vectors distributed as $Y$. Given
$p\geq 1$ and $x\in \C^p$, denote
$$
  Y(x)= Y(x, p):= \sum_{j=1}^p x_j Y^{(j)},
$$
where $Y^{(j)}$, $j\leq p$ are independent copies of $Y$.
%
%
In this subsection, we derive bounds on the order statistics of $Y$ and $Y(x)$.
We start with the following lemma.

\begin{lemma}[Small ball probability for order statistics]\label{l: orderstat}
Let $E\subset\C^n$ be a fixed subspace of $\C^n$, with $m:=\dim E^\perp$ bounded below by a large universal constant.
Then
$$
 \Prob\Big\{Y_{c m}^*\leq  \frac{c m }{n}\Big\}
 \leq \exp(-c \, m),
$$
where $c$ is a positive absolute constant.
\end{lemma}
\begin{proof}
Note that for every $i\leq n$ we have
$$
  Y^*_{i}\geq \min\big\{\|P_J(Y)\|_2/\sqrt{n}:\;J\subset[n],\;|J^c|=i\big\}.
$$
Therefore,
$$
\Prob\big\{Y^*_{i}\leq \tau\big\}\leq {n \choose i} \max\limits_{|J^c|=i}
\Prob\big\{\|P_J(Y)\|_2\leq \tau\sqrt{n}\big\}.
$$
Denoting $W= P_J \Proj_{E^\perp}$, and applying a small ball probability estimate for Gaussian vectors (\cite[Proposition~2.6]{LMOT}, see also Remark~\ref{rem: complex-anti-conc} below), we have
for every $\tau\leq c \Vert W\Vert_{\HS}/\sqrt{n}$,
$$
  \Prob\big\{\|P_J(Y)\|_2\leq \tau\sqrt{n}\big\}\leq \Big(\frac{\tau\sqrt{n}}{\Vert W\Vert_{\HS}}
  \Big)^{c'\frac{\Vert W\Vert_{\HS}^2}{\Vert W\Vert^2}},
$$
where  $c'\in(0, 1)$ is a universal constant. Note that $\Vert W\Vert\leq 1$ and
$$
\Vert W\Vert_{\HS}^2 =
\Tr\big(P_J \Proj_{E^\perp}\big)=  \Tr\big(\Proj_{E^\perp}\big)-\Tr\big(P_{J^c}
\Proj_{E^\perp}\big)\geq
m -\vert J^c\vert.
$$
Therefore for  $\tau< \Vert W\Vert_{\HS}/\sqrt{n}$ and   $i\leq c'm/4$
we have
\begin{align*}
 \Prob\big\{Y^*_{i}\leq \tau\big\}&\leq
  \Big(\frac{en}{i}\Big)^{i}
 \Big(\tau\sqrt{\frac{n}{m -i}}\Big)^{c' (m -i)}
 \leq  \Big(\frac{4en}{c' m}\Big)^{c m/4}   \Big(\tau\sqrt{\frac{2n}{m}}\Big)^{c' m/2}
\leq  \Big(\frac{8 n \tau}{\sqrt{c'}\, m}\Big)^{c' m/2}.
\end{align*}
The choice of $\tau= \sqrt{c'}\, m/(8 e n)$, $i= c' m/4$, and
$c=\min\{\sqrt{c'}/(8 e), c'/4\}$  completes the~proof.
\end{proof}

As a consequence of Lemma~\ref{l: orderstat}, we obtain a bound for
linear combinations.

\begin{prop}[Small ball for linear combinations]\label{l: orderstat unionbound}
Let $n\geq 1$ be large enough, $E\subset\C^n$ be a fixed subspace of $\C^n$
with $m:=\dim E^\perp \geq n^{1/2}$, and  $p\leq n^{1/4}$ be a positive integer.
Then
$$
 \Prob\big\{\inf\limits_{x}\;(Y(x))^*_{c_{\ref{l: orderstat unionbound}}m}
 \leq c_{\ref{l: orderstat unionbound}}\, m/n\big\}\leq
 \exp(-c_{\ref{l: orderstat unionbound}}m),
$$
where the infimum is taken over all complex $p$-dimensional unit vectors $x$
and $c_{\ref{l: orderstat unionbound}}>0$ is a universal constant.
\end{prop}

\begin{proof}
Let $\Net$ be a $c/(p n^2)$-net on the set of complex
unit vectors in $\C^p$ with cardinality $|\Net|\leq \big(3 pn^2/c\big)^{2p}$,
where  $c$ is the constant from Lemma~\ref{l: orderstat}.
Since for every $x$ the vector $Y(x)$ has the same distribution as $Y$,
Lemma~\ref{l: orderstat} together with  the union bound implies
\begin{align*}
\Prob\big\{\inf\limits_{x\in\Net}\;(Y(x))^*_{cm}\leq c m/n\big\}
&\leq |\Net|\exp\big(-cm\big)
\leq \exp\big(-cm+2p\ln(3 p n^2/c)\big).
\end{align*}
By the definition of $\Net$, for any unit vector $x\in\C^p$ there is $y=y(x)\in \Net$
such that $\|x-y\|_2\leq c/(p n^2)$, hence
$$
 \|Y(x)-Y(y)\|_2 =\Big\|\sum_{j=1}^p x_j Y^{(j)}-\sum_{j=1}^p y_j Y^{(j)}\Big\|_2
  \leq
  \sum_{j=1}^p | x_j - y_j|\,  \|Y^{(j)}\|_2\leq
  \frac{c}{n^2} \, \max_{j\leq p}\|Y^{(j)}\|_2.
$$
This immediately implies that
$$
  (Y(x))^*_{cm}\geq (Y(y))^*_{cm}
   -\frac{c}{n^2}\max_{j\leq p}\|Y^{(j)}\|_2.
$$
Thus, we obtain a deterministic relation
$$
  \inf\limits_{\|x\|_2=1}\;(Y(x))^*_{cm}\geq
  \inf\limits_{x\in\Net}\;(Y(x))^*_{cm} -
  \frac{c}{n^2}\max_{j\leq p}\|Y^{(j)}\|_2.
$$
This, together with a rough bound $\Prob\{\max\limits_{j\leq p}\|Y^{(j)}\|\geq n\}< e^{-n}$, yields
\begin{align*}
  \Prob\Big\{\inf\limits_{\|x\|_2=1}\;(Y(x))^*_{cm}\leq
   \frac{cm}{2n} \Big\}   &\leq
    \Prob\Big\{\inf\limits_{x\in\Net}\;(Y(x))^*_{cm}
    \leq  \frac{cm}{2n}+\frac{c}{n}\Big\}+e^{-n}\\
    &\leq \Prob\Big\{\inf\limits_{x\in\Net}\;(Y(x))^*_{cm}
    \leq cm/{n} \Big\}+e^{-n}\\
&\leq \exp\big(-cm + 2p\ln(3 p n^2/c)\big)+e^{-n}.
\end{align*}
Since $m\geq \sqrt{n}\geq p^2$, this completes the proof.
\end{proof}

We now pass to upper bounds.

\begin{lemma}[Large deviations of order statistics]\label{l: orderstat deviations}
Let $E$
be as in Lemma~\ref{l: orderstat}.
Then for every $i\leq n/2$ and  $\tau>0$ one has
$$
  \Prob\big\{Y_{i}^*\geq C\sqrt{\ln (n/i)}\big\}\leq \Big(\frac{i}{n}\Big)^i,
$$
where $C>0$ is  a universal constant.
\end{lemma}
\begin{proof}
Note that for a fixed $i\leq n$ we have
$$
 Y^*_i\leq \max\big\{\|P_J(Y)\|_2/\sqrt{i}:\;J\subset[n],\;|J|=i\big\}.
$$
Thus,
$$\Prob\{Y^*_i\geq\tau\}\leq {n\choose i}\cdot  \max\limits_{|J|=i}
\Prob\big\{\|W G\|_2\geq \tau\sqrt{i}\big\},
$$
where $W= P_J \Proj_{E^\perp}$. Using that $\Exp \|W G\|_2^2 = \Tr(W)\leq i$, we get
$$
\Prob\big\{\|W G\|_2\geq \tau\sqrt{i}\big\}\leq \Prob\big\{\|W G\|_2^2\geq \Exp\|W G\|_2^2 + (\tau^2-1)i\big\}.
$$
Applying  Hanson--Wright inequality (see for example \cite[Theorem~1.1]{RV} and Remark~\ref{rem: complex-anti-conc}),
we obtain that for any $\tau\geq \sqrt{2}$
$$
\Prob\big\{\|W G\|_2\geq \tau\sqrt{i}\big\}\leq \exp\big(-c \tau^2 i\big),
$$
for some absolute positive constant $c$. Taking $\tau= C \sqrt{\ln (n/i)}$ for sufficiently large
constant $C$, completes the proof.
\end{proof}

\begin{rem}\label{rem: complex-anti-conc}
The results of \cite{LMOT} and \cite{RV} used in this section are both formulated for real matrices and real random vectors. However, this is easily overcome by noticing that
if $W$ is an $n\times n$ complex matrix and $x\in \C^n$, then one may associate the $(2n)\times (2n)$ matrix
$$
\widetilde W=  \left[ {\begin{array}{cc}
   \Re\, (W) & -\Im\, (W) \\
   \Im\, (W) & \Re\, (W)\\
  \end{array} } \right] \quad \text{ and }\quad \widetilde x=  \left[ {\begin{array}{c}
   \Re\, (x)  \\
   \Im\, (x)  \\
  \end{array} }\right],
$$
where $\Re$ and $\Im$ denote the real and imaginary parts.  Now notice that $\Vert \widetilde W \widetilde x\Vert_2=\Vert W  x\Vert_2$ and thus $\Vert \widetilde W\Vert= \Vert W\Vert$. Moreover, one can check that $\Vert \widetilde W\vert_{\HS}^2= 2\Vert W\Vert_{\HS}^2$. Therefore, one could apply the results of \cite{LMOT} and \cite{RV} to $\widetilde W$ and deduce the analogous results for the complex case.
\end{rem}

As a  consequence of Lemma~\ref{l: orderstat deviations} we obtain a bound for
linear combinations.

\begin{prop}\label{p: orderstat combined}
Let $n$ be a large enough integer, $E$ be a fixed subspace of $\C^n$ with $m:=\dim E^\perp\geq n^{1/2}$, and
 $p\leq n^{1/4}$ be a positive integer.
Then
%
$$
  \Prob\big\{\sup\limits_{x}\;(Y(x))^*_i\geq C_{\ref{p: orderstat combined}} p\sqrt{\ln (n p/i)}
  \mbox{ for some }i\leq  c_{\ref{p: orderstat combined}}m\big\}\leq 2/\sqrt{n},
$$
where $C_{\ref{p: orderstat combined}},c_{\ref{p: orderstat combined}}$ are universal positive constants.
\end{prop}
\begin{proof}
Fix $i\leq n$ and a collection of $n$-dimensional vectors $\{z^1,z^2,\dots,z^p\}$.
Observe that for any subset $J\subset[n]$ of cardinality $i$, one has
\begin{align*}
\min\limits_{j\in J}|(z^1+\dots+z^p)_j|
\leq \min\limits_{j\in J}\sum\limits_{\ell=1}^p |z^\ell_j|
\leq p\min\limits_{j\in J}\max\limits_{\ell\leq p}|z^\ell_j|:=p\,a.
\end{align*}
For any $j\in J$ there is $\ell=\ell(j)\leq p$ such that $|z^\ell_j|\geq a$.
Hence, by the pigeonhole principle, there is $\ell_0\leq p$ such that
$|z^{\ell_0}_j|\geq a$ for at least $|J|/p=i/p$ indices from $J$.
Thus, we obtain
$$
  \min\limits_{j\in J}|(z^1+\dots+z^p)_j|\leq p\max\limits_{\ell\leq p}\; (z^\ell)_{\lceil i/p\rceil}^*.
$$
Note that the right hand side does not depend on the choice of $J$, therefore
$$
 (z^1+\dots+z^p)^*_i\leq p\max\limits_{\ell\leq p}\; (z^\ell)_{\lceil i/p\rceil}^*.
$$
Returning to vectors $Y^{(1)},\dots,Y^{(p)}$ we get
for any unit complex vector $x$,
$$
  (Y(x))^*_i\leq p\max\limits_{\ell\leq p}\; (x_\ell Y^{(\ell)})_{\lceil i/p\rceil}^*
  \leq p\max\limits_{\ell\leq p}\; (Y^{(\ell)})_{\lceil i/p\rceil}^*.
$$
Thus, denoting $m:=\dim E^\perp$ and applying Lemma~\ref{l: orderstat deviations},
we obtain for appropriate absolute constants $C\geq c>0$,
\begin{align*}
\Prob\big\{&\sup\limits_{x}\;(Y(x))^*_i\geq C p\sqrt{\ln (n p/i)}\mbox{ for some }i\leq cm\big\}\\
&\leq \Prob\big\{(Y^{(\ell)})_{\lceil i/p\rceil}^*\geq C\sqrt{\ln (n p/i)}\mbox{ for some }i\leq cm\mbox{ and }\ell\leq p\big\}\\
&\leq p\sum_{i=1}^{cm}\Big(\frac{\lceil i/p\rceil}{n}\Big)^{\lceil i/p\rceil}\leq\frac{2p^2}{n}\leq \frac{2}{\sqrt n},
\end{align*}
provided that $n$ is large enough. This completes the proof.
\end{proof}

\subsection{Strongly correlated indices}

Let $E$ be a fixed subspace of $\C^n$ and $Y=\Proj_{E^\perp} G$ as before.
Let $\alpha,\beta>0$ be parameters.
We say that a pair of indices $(i,j)$ is {\it $(\alpha,\beta)$-strongly correlated} (with respect to $E$) if
$$
 \Prob\{|Y_i-Y_j|\geq \alpha\}\leq\beta.
$$
Next, we inductively construct a sequence of (non-random) sets $(U_\ell)_{\ell\ge 1}=(U_\ell(\alpha, \beta))_{\ell\ge 1}$.
At the first step, choose $U_1$ as the largest subset of $[n]$ such that there is $u_1\in U_1$
so that $(u_1,u)$ is $(\alpha,\beta)$-strongly correlated for all $u\in U_1$. At the $\ell$-th step,
we define
$$
 U_\ell\subset \bar U_\ell := [n]\setminus (U_1\cup\ldots\cup U_{\ell-1})
$$
as the largest subset of $\bar U_\ell$,
such that there is an index $u_\ell\in U_\ell$ so that
$(u_\ell,u)$ is $(\alpha,\beta)$-strongly correlated for all $u\in U_\ell$ (if $\bar U_\ell=\emptyset$
then we set $U_\ell=\emptyset$ as well).
Further, it will be convenient for us to assume that the sequence $(U_\ell)_{\ell\ge 1}$
is uniquely defined. This can be achieved, for example, by defining a total order respecting cardinality
on the set of all subsets of $[n]$
and, at each step above, choosing the ``greatest'' admissible set with respect to that order.
Observe that, by the construction of $U_\ell$'s, the sequence of cardinalities $(|U_\ell|)_{\ell\ge 1}$
is non-increasing. Note that for every $\ell$ and for every $i, j\in U_\ell$, the pair
$(i,j)$ is $(2\alpha,2\beta)$-strongly correlated.

\begin{lemma}\label{l: sasha's}
Assume that a pair $(i,j)$ is not $(\alpha,\beta)$-strongly correlated for some $\alpha>0$ and $\beta\in(0,1/2]$. Then
for every $s>0$ one has
$$\Prob\big\{|Y_i-Y_j|\leq \alpha\, s/\sqrt{ \ln (1/\beta)}\big\}\leq s^2.$$
\end{lemma}

\begin{proof}
Set $\xi:=Y_i-Y_j$.
Observe that $\xi$ is a centered complex Gaussian variable and
denote its variance by $\sigma^2$. By the assumption of the lemma and
by (\ref{l: line projection}), we have
\begin{equation*}\label{eq: aux 39hgehg}
  \beta\leq  \Prob\{|\xi|>\alpha\} =    e^{-\alpha^{2}/\sigma^2},
\end{equation*}
which implies that
$
  \sigma\geq \alpha/\sqrt{ \ln (1/\beta)}.
$
Since for every $s>0$,
$$
 \Prob\{|\xi|\leq s \sigma\} = 1-e^{-s^2} \leq s^2
$$
 the desired result follows.
\end{proof}

The last lemma, combined with averaging arguments, implies the following lemma.

\begin{lemma}\label{l: Uells}
Let $\alpha>0$ and $\beta\in(0,1/2]$, and let the sequence $(U_\ell)_{\ell\ge 1}$
be defined
as above.
Let $k\geq 1$ and $b>0$ be such that $|U_k|\leq b$.
Then for every $s>0$ one has
$$\Prob\Big\{\exists \lambda\in\C:\;\Big|\Big\{j\in \bigcup_{\ell\ge k}
U_\ell:\;|Y_j-\lambda|\leq \alpha\,s/\sqrt{16\ln (1/\beta)}\Big\}\Big|\geq 2b\Big\}
\leq \frac{(sn)^2}{2b^2}.$$
\end{lemma}

\begin{proof}
First, observe that by the construction of  $(U_\ell)_{\ell\ge 1}$, for every
$i\in U:= \bigcup_{\ell\ge k} U_\ell$ we have
$$
 \big|\big\{j\in U
 \, :\, \, (i,j)\mbox{ are $(\alpha,\beta)$-strongly correlated}\big\}\big|\leq b.
$$
Hence, by Lemma~\ref{l: sasha's}, there is a non-random set $K_i\subset[n]$
such that  $|K_i|\le b$ and
$$
 \Prob\big\{|Y_i-Y_j|\leq \alpha\, s/\sqrt{ \ln (1/\beta)}\big\}\le s^2
$$
for all $s>0$ and $j\in U\setminus K_i$. Fix now $s>0$ and for every $i\in U$ define
the event
$$
  \Event_i:=\big\{\big|\big\{j\in U \, : \, \,
  |Y_i-Y_j|\leq \alpha\, s/\sqrt{ \ln (1/\beta)}\big\}\big|\geq 2b\big\}.
$$
Note that $\Event_i$ is contained in the event
$
 |\{j\in U \setminus K_i :  \,
  |Y_i-Y_j|\leq \alpha\, s/\sqrt{ \ln (1/\beta)}\big\}\big|\geq b.
$
Hence, applying Markov's inequality, we get
$$
 \Prob(\Event_i)\leq \frac{1}{b}\sum_{j\in U\setminus K_i}
 \Prob\big\{|Y_i-Y_j|\leq \alpha\, s/\sqrt{ \ln (1/\beta)}\big\} \leq \frac{s^2 n}{b}.
$$
Further,
\begin{align*}
\Prob&\Big\{\exists \lambda\in\C:\;\big|\big\{j\in U
 :\,\,  |Y_j-\lambda|\leq \alpha\,s/\sqrt{16\ln (1/\beta)}\big\}\big|\geq 2b\Big\}
\\&\leq
\Prob\Big\{\exists\,i\in U
 :\, \,
\big|\big\{j\in U:\;
|Y_i-Y_j|\leq \alpha\,s/\sqrt{ 4\ln (1/\beta)}\big\}\big|\geq 2b\Big\}\\
&\leq
\Prob\Big\{\exists\,J\subset U
 :\, \, |J|\geq 2b,\, \forall\,j_1\in J\, \, \, \,
\big|\big\{j\in U:\,
|Y_i-Y_{j_1}|\leq {\alpha\,s}/{\sqrt{ \ln (1/\beta)}}\big\}\big|\geq 2b\Big\}\\
&=\Prob\Big\{
\sum_{i\in U} \chi _{\Event_i}\geq 2b
\Big\}
\leq \frac{1}{2b}\sum_{i\in U}\Prob(\Event_i)\leq \frac{n}{2b}\cdot \frac{s^2 n}{b},
\end{align*}
where in the last formula we used Markov's inequality again.
This completes the proof.
\end{proof}

We will use all properties of Gaussian vectors established previously to show that if the number of strongly correlated pairs associated to $E$ is large, then we can construct an orthogonal vector to $E$ satisfying none of the assumptions of Theorem~\ref{t: kernel restated}, i.e., a normal vector to $E$ which is neither very steep nor sloping with many levels.

\begin{lemma}\label{l: strg corr to structure}
There exist universal constants $C_{\ref{l: strg corr to structure}}$ and $c_{\ref{l: strg corr to structure}}$ such that the following holds.
Let $\gamma>0$ and  $E$ be a fixed subspace of $\C^n$ with $m:=\dim E^\perp \geq n^{3/4}$. Denote
$$
  \alpha:=\exp\Big(-C_{\ref{l: strg corr to structure}} \Big(\frac{n}{m}\Big)^\gamma\Big),\quad \quad
  \beta:=\frac{1}{4}\Big(\frac{m}{4 n}\Big)^3, \quad  \quad \text{ and } \quad \quad V= 2\Big(\frac{m}{n}\Big)^{\gamma/2} n.
$$
Let the sequence $(U_\ell)_{\ell\ge 1}$
be defined
as above and $\ell_0\leq 4 n/m$. Suppose that $|\bigcup_{\ell=1}^{\ell_0-1}U_\ell|> V$.
Then there exists a vector $x\in \C^n$ orthogonal to $E$ such that
$$
  \forall i \leq  c_{\ref{l: strg corr to structure}}m  :\,  \, \, \,
  x_i^*\leq 0.9\, (n/i)^{3} x_{c_{\ref{l: strg corr to structure}} m}^*
$$
 and for all $\lambda\in\C$,
$$
\Big|\Big\{i\leq n:\,|x_i-\lambda |\leq \exp\big(- 2(n/m)^\gamma\big) x_{c_{\ref{l: strg corr to structure}} m}^* \Big\}\Big|> \Big(\frac{m}{n}\Big)^{\gamma/2} n.
$$
In other words,  there exists a vector $x\in E^\perp$, which is neither very steep nor sloping with many levels in the sense of Theorem~\ref{t: kernel restated}.
\end{lemma}
\begin{proof}
As before let
 $Y^1,\dots,Y^{\ell_0-1}$ be independent copies of the vector $Y$.
For any realization of $Y^1,\dots,Y^{\ell_0-1}$, let
$X=(X_1,X_2,\dots,X_{\ell_0-1})\in\C^{\ell_0-1}$ be a unit complex vector satisfying
$$
\forall \ell < \ell_0\, : \quad   \sum\limits_{k=1}^{\ell_0-1} X_k Y^k_{u_\ell}=\xi
$$
for some $\xi\in\C$
(we recall that $(u_\ell)_{\ell\ge 1}$ is the sequence of indices which was defined together with
the sequence of subsets $(U_\ell)_{\ell\ge 1}$). The vector $X$ can be found as follows: if the matrix $(Y^k_{u_\ell})_{1\leq k,\ell< \ell_0}$ is of full rank, then take the unique solution of the corresponding system with $\xi=1$ and normalize it, otherwise take any unit vector in the kernel of the matrix as $X$.
Set $Z:=\sum\limits_{k=1}^{\ell_0-1} X_k Y^k$.
Observe that deterministically
$$
   Z_{u_1}=Z_{u_2}=\dots=Z_{u_{\ell_0-1}}=\xi.
$$
We then have
\begin{align*}
\Prob&\Big\{|Z_j-Z_{u_1}|\geq \alpha \ell_0\mbox{ for at least half of indices }j\in \bigcup_{\ell=1}^{\ell_0-1}U_\ell\Big\}\\
&\leq \sum_{\ell=1}^{\ell_0-1}\Prob\big\{
|Z_j-Z_{u_\ell}|\geq \alpha \ell_0\mbox{ for at least half of indices }j\in U_\ell
\big\}\\
&\leq
\sum_{\ell=1}^{\ell_0-1}\sum_{k=1}^{\ell_0-1}\Prob\big\{
|Y^k_j-Y^k_{u_\ell}|\geq \alpha\;\mbox{for at least $|U_\ell|/(2\ell_0)$ indices }j\in U_\ell
\big\}\\
&\leq \sum_{\ell=1}^{\ell_0-1}\sum_{k=1}^{\ell_0-1}\, \frac{2\ell_0}{|U_\ell|}\, \sum_{j\in U_\ell} \Prob\big\{
|Y^k_j-Y^k_{u_\ell}|\geq \alpha\big\}
\leq 2{\ell_0}^3\,\beta,
\end{align*}
where the first inequality follows by the union bound; the second one by a combination
of the triangle inequality, the fact that $|X|=1$, and the union bound; the third one
from Markov's inequality; and the last one from the definition of $(\alpha,\beta)$-strongly
correlated pairs.
This together with the assumptions on $\ell_0$ and $\beta$ implies
$$\Prob\Big\{\exists\;\lambda\in\C:\;
|Z_j-\lambda|\leq \alpha \ell_0\mbox{ for more than $V/2$ indices }j\in [n]\Big\}\geq 1-2{\ell_0}^3\,\beta\geq 1/2.$$
On the other hand, applying Propositions~\ref{l: orderstat unionbound} and \ref{p: orderstat combined}
with $p=\ell_0-1$ we obtain that with probability at least $1-3/\sqrt{n}$
one has
$$Z^*_{c_{\ref{p: orderstat combined}}m}\geq c_{\ref{p: orderstat combined}}m/n$$
and
$$
  \forall i \leq  c_{\ref{l: strg corr to structure}}m  :\,  \, \, \,
  Z^*_i\leq C_{\ref{p: orderstat combined}} \ell_0\sqrt{\ln (n \ell_0/i)}.
$$
Intersecting the previous events
we deduce that there exists a realization of $Z$ satisfying
$$
 \Big|\Big\{i\leq n:\,|Z_i-\lambda|\leq \frac{\alpha \ell_0 n}{c_{\ref{p: orderstat combined}}m}
 \, Z^*_{c_{\ref{p: orderstat combined}}m}\Big\}\Big|>V/2
$$
for some $\lambda\in\C$, and,  using that $\ell_0\leq n/\dim E^\perp$, we have
$$
   \forall i \leq  c_{\ref{l: strg corr to structure}}m  :\,  \, \, \,
   Z_i^*\leq \frac{C_{\ref{p: orderstat combined}} n\,
  \ell_0\sqrt{\ln (n \ell_0/i)}}{c_{\ref{p: orderstat combined}}m}\,
  Z^*_{c_{\ref{p: orderstat combined}}m}
  \leq 0.9\, \Big(\frac{n}{i}\Big)^3 Z^*_{c_{\ref{p: orderstat combined}}m}.
$$
To complete the proof, note that
$$
   \frac{\alpha \ell_0 n}{c_{\ref{p: orderstat combined}}m}\leq
   \exp\Big(- 2\big(n/m\big)^\gamma\Big)
$$
for an appropriate choice of the constant $C_{\ref{l: strg corr to structure}}$.
\end{proof}

\subsection{Proof of Theorem~\ref{th: random normal structure}}

Let $d, n,  z, \gamma, I, E, G, A_n$ be as in the statement of Theorem~\ref{th: random normal structure}
and $Y$ as above.
We may assume without loss of generality that $\dim E= \vert I\vert$ a.s.,
otherwise, we complement $E$
to form a subspace $E_0$  of dimension $\vert I\vert$. In this case orthogonality to $E_0$ will imply orthogonality to $E$,
therefore the proof below won't be affected. Denote
%
%
\begin{align*}
&b=\frac{|I^c|}{\sqrt{2}}, \,   s=\frac{1}{\sqrt{2}}\Big(\frac{|I^c|}{n}\Big)^{3/2},  \,
 \alpha=\exp\Big(-C_{\ref{l: strg corr to structure}} \Big(\frac{n}{\vert I^c\vert}\Big)^\gamma\Big), \,
\beta=\frac{1}{4}\Big(\frac{|I^c|}{4 n}\Big)^3,  \,  V=2n\Big(\frac{\vert I^c\vert}{n}\Big)^{\gamma/2},
\end{align*}
where $C_{\ref{l: strg corr to structure}}$ is the constant  from Lemma~\ref{l: strg corr to structure}.
Let the sequence $(U_\ell)_{\ell\ge 1}$ be constructed as above.
Let $\ell_0\geq 1$ be the smallest index such that $|U_{\ell_0}|\leq b$. Since $(|U_\ell|)_{\ell\ge 1}$ is not increasing, then $\ell_0\leq 2n/b$.
Notice that $(U_\ell)_{\ell\ge 1}$ and $\ell_0$ inherits randomness only from $E$.
Let $\Event$ be the event (depending only on $E$) that
$$
  \Big|\bigcup_{\ell=1}^{\ell_0-1}U_\ell\Big|> V.
$$
Lemma~\ref{l: strg corr to structure} implies that
$\Event \subset \Event_1^c$,
where  $\Event_1$ denotes
the event appearing in Theorem~\ref{t: kernel restated}.
Denoting by $\Event_2$
the event of Theorem~\ref{th: random normal structure} and applying Theorem~\ref{t: kernel restated},
we get
$$
\Prob(\Event_2) \leq \Prob(\Event_2\cap \Event^c)
+\Prob(\Event)\leq \Prob(\Event_2\cap\Event^c)
+1/n.
$$
Now note that once in $ \Event^c$, one could take the set $\widetilde J$
in $\Event_2$ to be $\bigcup_{\ell=1}^{\ell_0-1}U_\ell$, which is of size
smaller than $V$. Therefore,
$$
\Prob(\Event_2\cap \Event^c)
\leq \Prob\Big\{\exists \lambda\in\C:\;\Big|\Big\{j\in \bigcup_{\ell\ge \ell_0}
U_\ell:\;|Y_j-\lambda|\leq \exp\Big(-C_{\ref{th: random normal structure}}\Big(\frac{n}{\vert I^c\vert}\Big)^\gamma\Big)\Big\}\Big|\geq 2b\Big\}.
$$
Since $n/\vert I^c\vert\geq d^3$ and $d$ is large enough, there exists
a sufficiently large absolute constant $C_{\ref{th: random normal structure}}$
satisfying
$$
  \exp\Big(-C_{\ref{th: random normal structure}}\Big(\frac{n}{\vert I^c\vert}
   \Big)^\gamma\Big)\leq \alpha\,s/\sqrt{16\ln (1/\beta)}.
$$
 Note that $|U_{\ell_0}|\leq b$. Therefore by Lemma~\ref{l: Uells}, we obtain
\begin{align*}
\Prob(\Event_2\cap \Event^c)
&\leq \Prob\Big\{\exists \lambda\in\C:\;\Big|\Big\{j\in \bigcup_{\ell\ge \ell_0}
U_\ell:\;|Y_j-\lambda|\leq \alpha\,s/\sqrt{16\ln (1/\beta)}\Big\}\Big|\geq 2b\Big\}\\
&\leq \frac{(sn)^2}{2b^2}=\frac{\vert I^c\vert}{2n},
\end{align*}
Putting together the above estimates and using that $\frac{1}{n}\leq  \frac{\vert I^c\vert}{2n}$ completes the proof.
\qed


\section{Intermediate singular values}\label{s: intermediate}

The goal of this section is to establish the  bounds on the intermediate singular values stated in the
introduction. We first briefly show how to derive the estimates on the singular values far from the
lower edge of the spectrum. As mentioned in the introduction, these follow from the work of Cook
\cite{Cook-circ}. The majority of the section is devoted to the more challenging regime,
that is, to bounding the singular values closer to the edge.

\subsection{Higher end of the spectrum}

Following the comparison strategy described in the introduction,
the following proposition was proved by Cook \cite[Proposition~7.3]{Cook-circ}.

\begin{prop}[{Anti-concentration of the spectrum}]\label{p:prop-cook-anti}
Assume $d=o(\sqrt{n})$ and $d\to \infty$ with $n$.
Then with probability at least $1-C_0\exp(-n)$ for all $\eta\in(0,1]$ one has
$$
  \nu_{B_z}
  ([0,\eta])< C_0(\eta+d^{-1/48}),
$$
where $C_0$ is an absolute positive constant.
\end{prop}
Based on this, it is easy to derive Proposition~\ref{prop2: cook-anti}.

\begin{proof}[Proof of Proposition~\ref{prop2: cook-anti}]
For  $i\leq n-2C'\, nd^{-1/48}$, set $\eta_i:=(n-i)/(2C'n)\geq d^{-1/48}$.
 Proposition~\ref{p:prop-cook-anti} applied with $\eta=\eta_i$ implies that with probability
 $1-\exp(-n)$, for any $i\leq n-2C_0\, nd^{-1/48}$ the number of singular values
smaller than $\eta_i$ is less than $2C_0\eta_i n$.
This yields that $s_i=s_{n-2C_0\eta_in} \geq \eta_i$.
Setting $C=2C'$ and $c=1/(2C_0)$ we complete the proof.
\end{proof}

\begin{rem}\label{rem: cook-smaller d}
Proposition~\ref{p:prop-cook-anti} was stated in \cite{Cook-circ} (see Proposition~7.3 there)
for $d\geq \ln^4 n$. Let us indicate the necessary changes needed to cover our range of interest,
that is, $d=o(\sqrt{n})$ and $d\to \infty$ with $n$ (without the restriction $d\geq \ln^4 n$).
Its proof combines three lemmas (see Lemmas~8.1, 8.2 and 8.4 in \cite{Cook-circ}). Lemma~8.4
establishes bounds on the intermediate singular values for shifts of Gaussian matrices and does not demand
$d$ to be polylogarithmic in $n$. \cite[Lemma~8.2]{Cook-circ} compares the expectation of the Stieltjes transforms
of the Bernoulli model (with parameter $d/n$) with its Gaussian counterpart. Here as well,
no restriction on $d$ is required and one only needs that $d\to \infty$ with $n$ for the
approximation to be effective. The last piece of the procedure, Lemma~8.1, compares
the uniform $d$-regular model with the Bernoulli matrix. Its proof  uses  a general
concentration inequality for linear eigenvalue statistics of Hermitian random matrices
\cite[Lemma~9.1]{Cook-circ} and an estimate of the probability that a Bernoulli matrix with parameter
$d/n$ is $d$-regular \cite[Lemma~9.2]{Cook-circ}. The latter indeed requires $d\geq \ln^4n$ as stated,
since it covers also large values of $d$. Since in our regime we suppose that $d=o(\sqrt{n})$,
we could replace the estimate of Lemma~9.2 by a bound proved by McKay and Wang \cite{MW03}, which
is also mentioned in Remark~9.3 in \cite{Cook-circ}. This implies the validity of Lemma~8.1 for
any $d=o(\sqrt{n})$ with the term $\exp(-O(d^{2/3}n\ln n))$ in the probability bound replaced with $\exp(-O(n\ln d))$.
This affects the proof of Proposition~7.3 in a trivial way, as one would change the choice of
$\varepsilon$ there to be $(\ln d /d)^{1/4}$ and carries the remaining part of the proof in
exactly the same manner. We note that the same change in Lemma~8.1 is sufficient to extend the proof
of Proposition~\ref{p:prop-conv-esv} to the wider range of $d$.
\end{rem}

\subsection{Lower end of the spectrum}

We first relate the intermediate singular values to separation estimates between the rows of the matrix.
As an important technical ingredient, we use the so-called negative second moment identity,
which was employed earlier in papers on the circular law (see \cite{TV10,Cook-circ}).
We note that one could also use the restricted invertibility principle instead (see \cite{NY}).

\begin{lemma}\label{l: Step II new}
Let $B$ be an $n\times n$ complex random matrix
with a distribution invariant under permutation of rows.
Let $m\leq n$ be positive integers and $\rho,\delta>0$ be such that
$$\Prob\big\{\d\big(\row_m(B),\spn_{j\leq m-1}\{\row_j(B)\}\big)<\rho\big\}\leq \delta.$$
Then for every $1\leq L\leq \frac{1}{2\delta}$ one has
$$\Prob\big\{s_{(1-2L\,\delta) m}(B)\geq \rho\sqrt{L\,\delta}\big\}\geq 1-\frac{1}{L}.$$
\end{lemma}
\begin{proof}
For each $i\leq m$, let $\chi_i$ be the characteristic function of the event
$$\big\{\d\big(\row_i(B),\spn_{j\in[m]\setminus\{i\}}\{\row_j(B)\}\big)<\rho\big\}.$$
By the conditions of the lemma (including the permutation invariance),
we have $\Exp\chi_i\leq \delta$, hence, by Markov's inequality, the event
$$\Event:=\Big\{\sum_{i=1}^m\chi_i> L\,\delta m\Big\}$$
has probability at most $1/L$.
Conditioning on the complement $\Event^c$,
we can find a set of indices $I\subset[m]$ of cardinality at least $m-L\,\delta m$
such that for every $i\in I$ one has
$$\d\big(\row_i(B),\spn_{j\in[m]\setminus\{i\}}\{\row_j(B)\}\big)\geq\rho.$$
Passing to the $|I|\times n$ submatrix $B'$ with rows
$\row_j(B)$, $j\in I,$
we obviously have for $i\leq |I|$,
$$\d\big(\row_i(B'),\spn_{j\neq i}\{\row_j(B')\}\big)\geq\rho.$$
Applying the negative second moment identity (see, e.g., \cite[Lemma~A.4]{TV10}), we obtain
$$
  \sum_{j=1}^{|I|}s_j(B')^{-2}=\sum_{j=1}^{|I|}
   \d\big(\row_i(B'),\spn_{j\neq i}\{\row_j(B')\}\big)^{-2}  \leq |I|\rho^{-2}.
$$
Therefore,
$$
 L\,\delta m\,s_{m-2L\,\delta m}(B')^{-2}\leq
 \sum_{j=m-2L\,\delta m}^{m-L\,\delta m}s_j(B')^{-2}
 \leq \sum_{j=1}^{|I|}s_j(B')^{-2}
 \leq m \rho^{-2},
$$
which implies
$$
 s_{m-2L\,\delta m}(B')\geq \rho\sqrt{L\,\delta}.
$$
Clearly, we deterministically have
$$s_{m-2L\,\delta m}(B)\geq s_{m- 2L\,\delta m}(B').$$
Thus, $s_{m-2L\,\delta m}(B)\geq \rho\sqrt{L\,\delta}$ everywhere on $\Event^c$,
which yields the desired result.
\end{proof}

We now provide bounds on the distances under consideration.

\begin{lemma}\label{l: step I}
Let $d, n$ be large enough integers such that $d\leq \ln^{96}n$,
$z\in \C$ be such that $\vert z\vert \leq \sqrt{d}\ln  d$,  $\gamma=1/288$,  and
$\sigma_n$ denote the uniform random permutation on $[n]$ independent of $\A$.
 Then for every
$$n-d^{-3} n\leq i\leq n- 2 n/\ln^{\gamma^{-1}} n$$ one has
\begin{align*}
 \Prob\Big\{&\d\big(\row_{\sigma_n(i)}(B_z),\, \spn_{j\leq i-1}
 \big\{\row_{\sigma_n(j)}(B_z)\big\}\big)
 <\exp\Big(-C\,\Big(\frac{n}{n-i}\Big)^{\gamma}\Big)\Big\}\leq C\, \frac{n-i}{n},
\end{align*}
where
$C$ is a positive universal constant.
\end{lemma}

\smallskip

In order to prove this lemma, we will develop specific anti-concentration
tools in the next subsection. We postpone its proof to the end of this
section and provide now the proof of Theorem~\ref{th: inter-sv}.

\begin{proof}[Proof of Theorem~\ref{th: inter-sv}]
Let $n-d^{-3} n\leq m\leq n- 2 n/\ln^{\gamma^{-1}} n$ and let $\sigma_n$, $B_z$, $C$
 be as in
Lemma~\ref{l: step I}. Then
we have
\begin{align*}
\Prob\Big\{&\d\big(\row_{\sigma_n(m)}(B_z),\, \spn_{j\leq m-1}
\big\{\row_{\sigma_n(m)}(B_z)\big\}\big)
<\exp\Big(-C\, \Big(\frac{n}{n-m}\Big)^{\gamma}\Big)\Big\}\leq C\, \frac{n-m}{ n}.
\end{align*}
 Let the matrix $B$ be obtained from the matrix $B_z$ by permuting its rows according to
 $\sigma _n$. Then  $B$ has the same singular values as $B_z$ and
the distribution of $B$  is invariant under permutation of rows.
 Therefore, we can apply Lemma~\ref{l: Step II new}  with
$$
 \rho=\rho(m)=\exp\Big(-C \Big(\frac{n}{n-m}\Big)^{\gamma}\Big), \quad
  \delta=\delta(m)=C\, \frac{n-m}{n} \quad \mbox{ and } \quad
  L=\frac{1}{2\sqrt{C \delta}}
$$
 to get that
$$
  \Prob\Big\{s_{(1-\sqrt{\varepsilon}) m}(B)\geq \Big(\frac{n-m}{4 n}\Big)^{1/2}
  \exp\Big(-C\Big(\frac{n}{n-m}\Big)^{\gamma} \Big)  \Big\}\geq 1-2\, C\, \sqrt{\frac{n-m}{n}},
$$
where we also denoted $\varepsilon= (n-m)/n$.
Using that  $\Big(\frac{n-m}{4 n}\Big)^{1/2}\geq \exp\Big(-\Big(\frac{n}{n-m}\Big)^{\gamma}\Big)$ when $d$
 is large enough (recall, $m\geq n-n/d^3$), we deduce that for an appropriate absolute constant $C_1>0$,
$$
  \Prob\Big\{s_{(1-2\sqrt{\varepsilon}) n}(B)\geq \exp\Big(-C_1/\varepsilon^\gamma \Big)  \Big\}\geq 1-2\,
  C \sqrt{\varepsilon},
$$
where we also used that $(1-\sqrt{\varepsilon}) m\geq (1-2\sqrt{\varepsilon}) n$. Writing $k=(1-2\sqrt{\varepsilon}) n$
(with slight adjustment to make it integer), so that $\eps = (n-k)^2/(2n)^2$, we clearly have
$$
 n-2d^{-3/2}n\leq k\leq n- 2\sqrt{2}n/\ln^{144} n.
$$
Using that $\sigma _n$ is independent of $A_n$,  $B$ and $B_z$ have the same
singular values, and that $(s_i)_i$ is increasing, we obtain the desired result.
\end{proof}

\subsection{Anti-concentration}\label{s: anti-conc}

To state the main result of the subsection, we need to define a special distribution
on the set of $n$-dimensional $0/1$ vectors.
For any matrix $M\in\Mc$ and for any non-empty subset $T\subset[n]$
denote
 by $\mathbb M_{M,T}$ the set of all matrices $M'$ in $\Mc$ satisfying
$$
  \mathbb M_{M,\sigma}:= \big\{ M'\in\Mc \, : \, \,
  \row_i(M')=\row_i(M) \quad \mbox{ for all } \quad
   i\notin T  \big\}.
$$

Now, fix $J\subset[n]$ of cardinality at least $n/2$.
In this section, we denote  by
$
 \mathcal I = \mathcal I (J)
$
 a uniform random subset of $J$
with cardinality $\lfloor n^{1/4}\rfloor$.
Next, fix an index $u$ and a matrix $M\in\Mc$ and define a random vector $X_{M,J,u}$
via its conditional distribution with respect to $\mathcal I$; namely, we postulate that,
conditioned on a realization $I_0$ of the set $\mathcal I$, the vector $X_{M,J,u}$ takes values in the set
$$
 Q_{M,J,u}:=\{\row_{u}(M'):\,M'\in\mathbb M_{M,I_0\cup\{u\}}\}
$$
and
$$
  \forall x\in Q_{M,J,u}\, : \, \, \, \Prob\big\{X_{M,J,u}=x\,|\,\mathcal I=I_0\big\}
  =\frac{|\{M'\in \mathbb M_{M,I_0\cup\{u\}}:\,\row_{u}(M')=x\}|
  }{|\mathbb M_{M,I_0\cup\{u\}}|}.
$$

\begin{prop}\label{prop: anti-c}
Let $d, n$ be large enough positive integers such that $d\leq n^{1/8}$.
Let $J$ be a subset of $[n]$ of cardinality at least $n/2$, $u\in[n]\setminus J$, and let $M$ be a fixed matrix in $\Mc$.
Further, let $\delta,\rho>0$,  $y$ be a fixed vector in $\C^n$ such that for some subset $\widetilde J\subset[n]$
we have
$$\forall\,\lambda\in\C:\quad \big|\big\{j\in [n]\setminus \widetilde J:\, |y_j-\lambda|\leq \rho\big\}\big|\leq \delta n.$$
Then,
$$
  \forall\,\lambda\in\C:\quad \Prob\big\{|\langle y,X_{M,J,u}\rangle-\lambda|\leq \rho/4\big\}
  \leq \big(8|\widetilde J|/n\big)^d+144\delta+n^{-1/10}.
$$
\end{prop}

\medskip

To prove this proposition we need several lemmas.

\begin{lemma}\label{l: I disjoint}
Let $d, n$ be large enough positive integers such that $d\leq n^{1/8}$
and $M\in \Mc$ be a fixed matrix.
Further, let $J\subset[n]$ be a fixed subset of cardinality at least $n/2$,  $u\in [n]\setminus J$ and
$\mathcal I=\mathcal I (J)$.
Then with probability at least $1-2 n^{-1/4}$ the supports of the rows $\row_i(M)$, $i\in\mathcal I\cup\{u\}$,
are pairwise disjoint.
\end{lemma}

\begin{proof}
Denote by $Q\subset (J\cup\{u\})\times (J\cup\{u\})$ the subset of all pairs $(i,j)$ such that
$$
 \supp\row_i(M)\cap\supp\row_j(M)\neq\emptyset.
$$
By $d$-regularity we observe that for any $i\in J\cup\{u\}$ there are less than $d^2$ indices $j$
with $(i,j)\in Q$. Thus, $|Q|\leq d^2 (|J|+1)$. On the other hand, an easy computation shows that
for any pair $(i_1, i_2)\in Q$ with $i_1\neq i_2$, the probability that
both $i_1$ and $i_2$ belong to $\mathcal I$, is equal to
$${|J|+1-2\choose \lfloor n^{1/4}\rfloor-2}\,{|J|+1\choose \lfloor n^{1/4}\rfloor}^{-1}
=\frac{\lfloor n^{1/4}\rfloor\, (\lfloor n^{1/4}\rfloor-1)}{|J|(|J|+1)}.$$
Hence,
$$
  \Prob\{\mathcal I\mbox{ contains a disjoint pair in }Q\}\leq |Q|\sqrt{n}/(|J|(|J|+1))\leq d^2\sqrt{n}/|J|.
$$
The assumptions on $|J|$ and $d$ imply the result.
\end{proof}

\smallskip

\begin{lemma}\label{l: I intersection}
Let $d\leq n$ be large enough positive integers
and $M\in \Mc$ be a fixed matrix. Further, let $J\subset[n]$ be subset of cardinality at least $n/2$, and let
$\mathcal I=\mathcal I(J)$.
Then for every subset $L\subset[n]$ with probability at least $1-1/n^2$ we have
$$
  \Big|\Big(\bigcup\limits_{i\in\mathcal I}\supp\row_i(M)\Big)\cap L\Big|\leq
  14 d^2 \ln n+4 d n^{-3/4}|L|.
$$
\end{lemma}

\begin{proof}
Fixing a partition $(L_k)_{k=1}^{d^2}$ of $L$ such that for every $k\leq d^2$
and $i\neq j\in L_k$, there is no row of $M$ such that $i,j$ are simultaneously contained in its support.
Such a partition can be constructed as follows: take an auxiliary graph $\Gamma$ on $L$ without loops such that
$i\ne j\in L$ are connected by an edge whenever
there is a row of $M$ whose support contains both $i$ and $j$.
The $d$-regularity immediately implies that the
maximum vertex degree of this graph is strictly less than $d^2$
(in fact, not greater than $d(d-1)$). Therefore, by Brook's theorem,  the chromatic number of $\Gamma$
does not exceed $d^2$, which justifies the number of sets in the required partition of $L$.

Further, let $\widetilde{\mathcal I}$ be a random subset of $J$, such that each index $i\in J$
is included into $\widetilde {\mathcal I}$ with probability $\lfloor n^{1/4}\rfloor/|J|$
independently of the others.
Fix for a moment  $k\leq d^2$.
For any $i\in L_k$, let $\eta_i^k$ be the indicator function of the event that
$$
 i\in \bigcup\limits_{j\in\widetilde{\mathcal I}}\supp\row_j(M).
$$
Note that by our construction $(\eta_i^k)_{i\in L_k}$ are jointly independent and that for all $i\in L_k$
$$
 \Exp\,\eta_i^k=\Exp\,(\eta_i^k)^2=\Prob\{\eta_i^k=1\}\leq d n^{1/4}/|J|:=\delta.
$$
Applying Bernstein's inequality with $t=\delta |L_k| + 14 \ln n$, we obtain
\begin{align*}
\Prob\Big\{\Big|L_k\cap \bigcup\limits_{j\in\widetilde{\mathcal I}}\supp\row_j(M)\Big| &\geq
2\delta |L_k| + 14 \ln n\Big\}
\leq \Prob\Big\{\sum_{i\in L_k} (\eta_i^k -\Exp\, \eta_i^k) \geq  t\Big\}\\
\\&\leq
\exp\Big(-\frac{3 t^2}{2(t + 3\delta |L_k|)}\Big)\leq \exp\Big(-\frac{3 t}{8}\Big)\leq n^{-5}.
\end{align*}
Then the union bound implies that with probability at least $1-d^2 n^{-5}$ one has
$$
 \Big|L\cap \bigcup\limits_{j\in\widetilde{\mathcal I}}\supp\row_j(M)\Big|\leq
  \sum\limits_{k=1}^{d^2}\Big( 14  \ln n+\frac{2dn^{1/4}|L_k|}{|J|}\Big)=14 d^2 \ln n + \frac{2dn^{1/4}|L|}{|J|}.
$$
Finally note that the cardinality of $\widetilde {\mathcal I}$
equals exactly $m:=\lfloor n^{1/4}\rfloor$ with probability
$$
 {|J|\choose m}
 \Big(\frac{m}{|J|}\Big)^{m}\,
 \Big(1- \frac{m}{|J|}\Big)^{|J|-m}\geq \Big(1- \frac{m}{|J|}\Big)^{|J|}\geq
 \exp(- 2m)\geq n^{-1/4}.
$$
Therefore
$$
  \Prob\Big\{\Big|L\cap \bigcup\limits_{j\in\widetilde{\mathcal I}}\supp\row_j(M)\Big|
  \leq 14 d^2 \ln n+\frac{2dn^{1/4}|L|}{|J|}\;\;\Big|\;\;|\widetilde{\mathcal I}|=\lfloor n^{1/4}\rfloor\Big\}
  \geq 1-d^2 n^{-4}\geq 1-\frac{1}{n^2},
$$
which implies the desired result, since $|J|\geq n/2$.
\end{proof}

\smallskip

\begin{lemma}\label{lem: dist-X-MJu}
Let $d<n$ be positive integers.
Let $M\in \Mc$ be a fixed matrix, $J\subset[n]$ be a subset of cardinality at least $n/2$, and
$\mathcal I=\mathcal I(J)$.
Let $u\in[n]\setminus J$ and $I_0\subset J$ of size $\lfloor n^{1/4}\rfloor$ be such that
the supports of the rows $\row_i(M)$, $i\in I_0\cup\{u\}$, are pairwise disjoint.
Then, conditioned on $\mathcal I=I_0$, the support of the random vector $X_{M,J,u}$ is
a uniformly distributed $d$-subset of
$$
  S:=\bigcup_{i\in I_0\cup\{u\}}\supp\, \row_i(M).
$$
\end{lemma}

\begin{proof}
We  first show that for any two $0/1$ vectors $x,y$ satisfying
$$
  \supp\, x,\supp\, y\subset S,
  \, \, \,
 |\supp\, x|= |\supp\, y|=d,\, \, \, \mbox{ and }\, \, \, |\supp\, x\setminus \supp\, y|=1,
$$
the sets
$$
 S_x := \big\{M'\in \mathbb M_{M,I_0\cup\{u\}},\,\row_u(M')=x\big\}\quad \mbox{ and } \quad
 S_y:= \big\{M''\in \mathbb M_{M,I_0\cup\{u\}},\,\row_u(M'')=y\big\}
$$
have the same cardinality.
Without loss of generality, assume that $x_1=y_2=1$ and $x_2=y_1=0$. Then $\{1, 2\} \subset S$.
For every matrix $M'\in S_x$
we  construct a matrix $M''\in S_y$
as follows. Since $\{1, 2\} \subset S$ and the rows indexed by $I_0\cup\{u\}$ are pairwise disjoint,
there exists a unique index  $i=i(M')\in I_0\cup\{u\}$ such that $M_{i,1}'=0$ and
$M_{i,2}'=1$.
Let $M''$ be obtained by performing the simple switching operation on $M'$
which interchanges the entries $M_{u,1}'$ and $M_{u,2}'$ with $M_{i,1}'$ and $M_{i,2}'$ respectively.
Clearly $M''\in S_y$, moreover, it is not difficult to see that the constructed mapping is injective.
Therefore, $|S_x|\leq |S_y|$.
Reversing the argument, we get that $|S_x|= |S_y|$. Since for every $0/1$ vector $z$ satisfying
$\supp\, z\subset S$ and  $|\supp \, z|=d$ one can construct a sequence of vectors
$x_0=x, x_1, \ldots, x_k=z$ with $\supp\, x_i\subset S$,  $|\supp \, x_i|=d$, and such that two
vectors $x_{i-1}$, $x_i$ differ on exactly two coordinates for every $1<i\leq k$, we obtain $|S_x|=|S_z|$.
Thus
$$
 \Prob\big\{X_{M,J,u}=x\,|\,\mathcal I=I_0\big\}=\Prob\big\{X_{M,J,u}=z\,|\,\mathcal I=I_0\big\},
$$
which means that, conditioned on $\mathcal I=I_0$, the support of the random vector
$X_{M,J,u}$ is uniformly distributed on the set of $d$-subsets of $S$.
\end{proof}

\begin{lemma}[Coupling]\label{l: XY coupling}
Let $d, n$ be large enough positive integers such that $d\leq n^{1/8}$
and  $M\in \Mc$ be a fixed matrix. Let $J\subset[n]$ be a subset of cardinality at least $n/2$ and
$\mathcal I=\mathcal I(J)$. Assume that
$u\in[n]\setminus J$ and let $I_0\subset J$ be of size $\lfloor n^{1/4}\rfloor$ and such that the supports of rows
$\row_i(M)$, $i\in I_0\cup\{u\}$, are pairwise disjoint.
Let $\xi_1,\dots,\xi_d$ be i.i.d.\ random variables
uniformly distributed on
$$
  S:=\bigcup\limits_{i\in I_0\cup\{u\}}\supp\,\row_i(M), \quad \mbox{ and set } \quad Y_\xi:=\sum_{i=1}^d e_{\xi_i}.
$$
Then there is a coupling $(X,Y_\xi)$, with $X$ distributed as
$X_{M,J,u}$, such that, conditioned on $\mathcal I=I_0$, we have
$$
 \Prob\big\{X=Y_\xi\,|\,\mathcal I=I_0\big\}\geq 1-n^{-1/8}.
$$
\end{lemma}

\begin{proof}
Note that, conditioned on the event
$$
  \Event :=\{\forall i\ne j \, \, \mbox{ one has }\, \, \xi_i\neq \xi_j \},
$$
the random set $X:=\{\xi_1,\ldots,\xi_d\}$ is a uniformly distributed $d$-subset of $S$.
Therefore, by  Lemma~\ref{lem: dist-X-MJu}, the distribution of $X_{M,J,u}$ conditioned on
$I=I_0$
agrees with the  distribution of $Y_\xi$ conditioned on $\Event$.
Since  $\row_i(M)$, $i\in I_0\cup\{u\}$, are pairwise disjoint, we have $|S|\geq d n^{1/4}$, hence
$$
 \Prob\{\xi_i=\xi_j\, \, \mbox{ for some }\, \, i\neq j\}\leq d^2 \,
 \Prob\{\xi_1=\xi_2\}\leq d^2/ |S|\leq  n^{-1/8}.
$$
This implies the desired result.
\end{proof}

\smallskip

\begin{lemma}\label{l: acnc partition}
Let $\delta, \rho >0$, $\widetilde J\subset[n]$, and $y$ be a fixed vector in $\C^n$
such that
$$
  \forall\,\lambda\in\C:\quad \big|\big\{j\in [n]\setminus \widetilde J:\,
  |y_j-\lambda|\leq \rho\big\}\big|\leq \delta n.
$$
Then there exists a partition $(U_{ij})_{i\leq 9, \,j\leq n}$ of $[n]\setminus \widetilde J$
such that $\vert U_{ij}\vert\leq \delta n$ for all $i\leq 9,\, j\leq n$, and
$$
 \forall i\leq 9 \, \, \,\, \,\forall j\neq j'\in [n] \,\, \, \, \,\forall s\in U_{ij}  \,\, \, \, \,\forall s'\in U_{ij'}
 \, \, \, \, \, \,\, \,\quad  \vert y_s-y_{s'}\vert\geq \rho.
$$
\end{lemma}
\begin{proof}
We identify $\C$ with $\R^2$. Consider the following nine points
\begin{align*}
&a_1=(0,0),\, a_2=(1,0),\, a_3=(2,0),\,
a_4=(0,1),\, a_5=(0,2),\\
&a_6=(1,1),\, a_7=(2,1),\, a_8=(1,2),\, a_9=(2,2).
\end{align*}
 For $i\leq 9$, set
$$
 \mathcal V_i:= \rho(a_i + 3\Z\times 3\Z).
$$
Note that any two points in $\mathcal V_i$ are at distance at least $3\rho$ and that the union of
$\mathcal V_i$'s is $\C$. We first construct a partition $(\mathcal V_{ij})_{i\leq 9,\, j\in \Z^2}$ of
the complex plane as follows. First, set
$\mathcal V_{1j}$'s to be the Euclidean balls of radius $\rho$ centered at $\rho(a_1 + 3j)\in\mathcal V_1$.
Observe that the balls are necessarily pairwise disjoint.
Further, assuming that $\mathcal V_{\ell j}$, $\ell<i$, $j\in \Z^2$ are constructed (for some $1<i\leq 9$),
define $\mathcal V_{i j}$ as the set difference of the Euclidean ball of radius $\rho$ centered at $\rho(a_i + 3j)\in\mathcal V_i$,
and the union of $\mathcal V_{\ell j'}$, $\ell<i$, $j'\in \Z^2$. Then $(\mathcal V_{ij})_{i\leq 9,\, j\in \Z^2}$ is a partition and
moreover, for any $i\leq 9$ and any $j\neq j'\in \Z^2$, one has $\vert x-x'\vert \geq  \rho$
for any $x\in \mathcal V_{ij}$, $x'\in \mathcal V_{ij'}$. Indeed, this follows by an application of the triangle inequality
together with the fact that the centers of these two balls are at distance at least $3\rho$.
Therefore, one can partition the coordinates of $y$ by intersecting the above partition of $\C$ with $\{y_i\}_{i\leq n}$.
This naturally defines a partition of $[n]\setminus \widetilde J$ by setting the  sets of the partition to be the indices of the corresponding
coordinates of $y$. The assumption on $y$ implies that each set in the partition contains at most $\delta n$ elements.
\end{proof}

\begin{proof}[{Proof of Proposition~\ref{prop: anti-c}}]
Fix $\lambda\in\C$. Then
$$
 \Prob\big\{|\langle y,X_{M,J,u}\rangle-\lambda|\leq \rho/4\big\}
 \leq \sum\limits_{I_0\subset J,\atop |I_0|=\lfloor n^{1/4}\rfloor}
 \Prob\big\{|\langle y,X_{M,J,u}\rangle-\lambda|\leq \rho/4\,|\,\mathcal I=I_0\big\}\,
 \Prob\big\{\mathcal I=I_0\big\}.
$$
Let $(U_{ij})_{i\leq 9, \,j\leq n}$
be the partition of $[n]\setminus \widetilde J$ given by Lemma~\ref{l: acnc partition},
in particular $|U_{ij}|\leq \delta n$ for all $i, j$.
Let $T$ be the collection of all subsets $I_0$ of $J$ of cardinality $\lfloor n^{1/4}\rfloor$
satisfying the following three conditions:
%
%
%
%
%
%
%
%
%
\begin{equation}\label{prop-T1}
  \mbox{ the rows } \, \, \, \row_i(M),\, \, \, \mbox{ for }\, \, \,  i\in I_0\cup\{u\} \, \, \,
  \mbox{  are pairwise disjoint; }
\end{equation}
\begin{equation}\label{prop-T2}
  \Big| \widetilde J\,\cap \, \bigcup_{i\in I_0}\supp\,\row_i(M)\Big| \leq 14 d^2 \ln n+ 4 d n^{-3/4}|\widetilde J|;
\end{equation}
\begin{equation}\label{prop-T3}
  \Big|U_{ij}\, \cap \,\bigcup_{i\in I_0}\supp\,\row_i(M)\Big|\leq 14 d^2 \ln n+  4 \delta dn^{1/4}.
\end{equation}
By Lemmas~\ref{l: I disjoint}, \ref{l: I intersection} and the union bound, the event
$\{\mathcal I\in T\}$ has probability at least $1-3 n^{-1/4}$.
Thus, we have
$$\Prob\Big\{|\langle y,X_{M,J,u}\rangle-\lambda|\leq \frac{\rho}{4}\Big\}
\leq \sum\limits_{I_0\in T}
\Prob\Big\{|\langle y,X_{M,J,u}\rangle-\lambda|\leq \frac{\rho}{4}\,\,\big|\,\,\mathcal I=I_0\Big\}\,
\Prob\Big\{\mathcal I=I_0\Big\}+\frac{3}{n^{1/4}}.$$
Further, fix any $I_0$ in $T$. Let $S$,  $\xi_1,\ldots,\xi_d$, and
$Y_\xi$ be defined in Lemma~\ref{l: XY coupling}. Note that by (\ref{prop-T1}), $|S|\geq d n^{1/4}$.
Lemma~\ref{l: XY coupling} implies
$$
  \Prob\big\{|\langle y,X_{M,J,u}\rangle-\lambda|\leq \rho/4\,\, \big|\,\,\mathcal I=I_0\big\}
  \leq\Prob\big\{|\langle y,Y_\xi\rangle-\lambda|\leq \rho/4\big\}+n^{-1/8}.
$$
Denote
$$
 S_0:= \bigcup_{i\in I_0}\supp\,\row_i(M)\setminus \widetilde J, \quad  \quad S_1:= \Big(\widetilde J\, \cap\,
  \bigcup_{i\in I_0}\supp\,\row_i(M) \Big) \, \cup\, \supp\, \row_u(M) ,
$$
and
$\xi=\{\xi_1,\ldots,\xi_d\}.$
Note that by properties (\ref{prop-T1}) and (\ref{prop-T2}) and assuming that $|\widetilde J|\leq n/8$
(otherwise the bound for the probability in Proposition~\ref{prop: anti-c} is trivial), one has
\begin{equation}\label{ver}
  |S|\geq d n^{1/4}, \quad \frac{|S_1|}{|S|} \leq
   \frac{15 d\ln n}{n^{1/4}} + \frac{4 \vert \widetilde J\vert}{n}\leq
   \frac{3}{4}, \quad \mbox{ and } \quad  \frac{|S_0|}{|S|} = 1- \frac{|S_1|}{|S|}\geq \frac{1}{4}.
\end{equation}
Consider two events
$$
   \Event_1 :=\{ \xi \cap S_0 = \emptyset\} =\{\xi \subset S_1\}
   \quad \mbox{ and } \quad \Event_2 := \{\xi \cap S_0 \ne \emptyset\} .
$$
Using property (\ref{ver}) and independence of $\xi_i$'s, we have
we clearly have
$$
  \Prob(\Event_1) = \left(|S_1|/|S|\right)^d
  \leq  \Big(\frac{30 d\ln n}{n^{1/4}}\Big)^d +
  \Big(\frac{8 \vert \widetilde J\vert}{n}\Big)^d .
$$
To estimate the remaining probability
we split $\Event_2$ into  disjoint union of events
$$
    \Event_W :=\{\xi_{i}\in S_0 \, \,  \mbox{ for all } \, \,  i\in W \quad \mbox{ and } \quad
    \xi_{i}\notin S_0 \, \,  \mbox{ for all } \, \,  i\notin W\},
$$
where $W$ runs over all non-empty subsets of $[d]$. Then
$$
 \Prob\{|\langle y,Y_\xi\rangle-\lambda|\leq \rho/4\, \, \big|\, \,  E_2\}\leq
 \sup_W \Prob\{|\langle y,Y_\xi\rangle-\lambda|\leq \rho/4 \, \, \big|\, \, \Event_W\}.
$$
Fix a non-empty $W\subset [d]$ and $m\in W$. Using that $\xi_i$'s are i.i.d. we observe
that
\begin{align*}
 \Prob\{|\langle y,Y_\xi\rangle-\lambda|\leq \rho/4 \, \, \big|\, \, \Event_W\} &\leq
 \sup \limits_{\widetilde\lambda\in \C}
 \Prob\big\{
  |\langle y,e_{\xi_1}\rangle-\widetilde\lambda|\leq \rho/4 \, \, \big|\, \,    \xi_m\in S_0 \big\}
 \\ &
  =\sup \limits_{\widetilde\lambda\in \C}
 \Prob\big\{
  |\langle y,e_{\xi_1}\rangle-\widetilde\lambda|\leq \rho/4  \, \, \big|\, \,   \xi_1\in S_0 \big\}.
\end{align*}
This implies
%
%
\begin{align*}
 p_0:=  \Prob\big\{\Event_2 \, \, \mbox{ and } \, \, |\langle y,Y_\xi\rangle-\lambda|\leq \rho/4\big\}
 \leq
 \sup \limits_{\widetilde\lambda\in \C}
 \Prob\big\{
  |\langle y,e_{\xi_1}\rangle-\widetilde\lambda|\leq \rho/4  \, \, \big|\, \,   \xi_1\in S_0 \big\}.
\end{align*}
Fix $\widetilde\lambda\in \C$. By Lemma~\ref{l: acnc partition} for every $i\leq 9$ there exists at most
one $j(i)\leq n$ such that
$$
  \xi_1\in S_0 \, \, \mbox{ and } \, \, |\langle y,e_{\xi_1}\rangle-\widetilde\lambda|
   \leq \rho/4   \quad \mbox{ implies  } \quad   \xi_1 \in S_0 \cap \bigcup _{i=1}^9  U_{i j(i)}.
$$
Using this, (\ref{prop-T3}) and (\ref{ver}), we observe
$$
   p_0\leq \frac{1}{\Prob\left(\xi_1\in S_0\right)}\, \sum _{i=1}^9 \Prob\big\{\xi_1\in S_0 \cap  U_{i j(i)} \big\} \leq
  \frac{|S|}{|S_0|}\,
  \sum _{i=1}^9 \frac{ |S_0 \cap U_{i j(i)}|}{|S|} \leq \frac{540\, d \ln n}{n^{1/4}} +
   144 \delta  .
$$
Since $\Prob\big\{|\langle y,Y_\xi\rangle-\lambda|\leq \rho/4\big\}  \leq \Prob(\Event_1) + p_0$,
$d\leq n^{1/8}$, and $n$ is large enough, this completes the proof.
\end{proof}

\subsection{Distances estimates}

The goal of this subsection is to prove Lemma~\ref{l: step I}.

Fix $z\in \C$, $\gamma= 1/(288)$,  and $i\in [n]$ satisfying
$n/\ln^{\gamma^{-1}}n \leq n-i \leq d^{-3}n$.
Recall that  $\sigma_n$ denotes the uniform random
permutation on $[n]$ independent of $A_n$ and $B_z=d^{-1/2}\A-z\idmat$.
Denote  $E_i:= E(B_z, \sigma ([i-1])$, i.e., the random subspace
spanned by the rows $\row_{\sigma_n (j)}(B_z)$, ${j\leq i-1}$.

\smallskip

We now define a random triple $(A_n,A_n',\sigma_n)$ in the following way
(the choice of notation will be justified after construction).
For each matrix $M\in\Mc$ and a permutation $\sigma\in\Pi_n$ let
$$
  \mathbb M_{M,\sigma}:= \big\{ M'\in\Mc \, : \, \,
  \row_{\sigma(j)}(M')=\row_{\sigma(j)}(M) \quad \mbox{ for all } \quad
   j\not\in [i-\lfloor n^{1/4} \rfloor, \, i]   \big\}.
$$
Define the set
$$
 U:=\bigcup_{\sigma\in\Pi_n}\bigcup_{M\in\Mc}\big\{(M,M',\sigma):\,M'\in \mathbb M_{M,\sigma}\big\}.
$$
Further,  define a probability measure $\eta$ on $U$ by
$$
 \forall (M,M',\sigma)\in U  : \, \, \,\,  \eta\big(\big\{(M,M',\sigma)\big\}\big)
 =\frac{1}{n!\,|\Mc|}\frac{1}{|\mathbb M_{M,\sigma}|}.
$$
We postulate that the triple $(A_n,A_n',\sigma_n)$ takes values in $U$ and is distributed according to the measure $\eta$.
It is not difficult to see that (individual) marginal distributions of $A_n$ and $A_n'$ are uniform on $\Mc$,
and that $\sigma_n$ is uniformly distributed on $\Pi_n$. Moreover, $A_n$ and $\sigma_n$ are independent,
as well as $A_n'$ and $\sigma_n$. This justifies our choice of notation for $A_n$ and $\sigma_n$
(which otherwise would come into conflict with our ``old'' notions of $A_n$ and $\sigma_n$).
As usual, below we assume that $G$ is independent from the triple $(A_n,A_n',\sigma_n)$ and that all random
variables are defined on the same probability space.

Fix a matrix $M\in\Mc$, a subset $J\subset[n]$ of cardinality $i-1$ and an index $u\in[n]\setminus J$.
Define the event
$$\Event_{M,J,u}:=\Big\{A_n=M,\;\{\sigma_n(r):\,r\leq i-1\}=J,\;\sigma_n(i)=u\Big\}.$$
Observe that, conditioned on $\Event_{M,J,u}$, the set
$$
 W:=\{\sigma_n(j):\,j=i-\lfloor n^{1/4}\rfloor,\dots,i-1\}
$$
is a uniform random $\lfloor n^{1/4}\rfloor$-subset of $J$.
Let $W_0\subset J$ be any realization of $W$ and set
$$\Event_{M,J,u,W_0}:=\Event_{M,J,u}\cap\big\{W=W_0\big\}.$$
Conditioned on $\Event_{M,J,u,W_0}$,
$A_n'$ takes values in the set of matrices
$\mathbb M_{M,W_0\cup\{u\}}$ defined the same way as in Section~\ref{s: anti-conc},
and the $u$-th row of $A_n'$ has conditional distribution defined by
$$\Prob\big\{\row_{u}(A_n')=x\,|\,\Event_{M,J,u,W_0}\big\}
=\frac{|\{M'\in \mathbb M_{M,W_0\cup\{u\}}\,:\,\row_u(M')=x\}|}{|\mathbb M_{M,W_0\cup\{u\}}|}.$$
In other words, conditioned on $\Event_{M,J,u}$, the $u$-th row of $A_n'$ is distributed exactly the same way as the random vector
$X_{M,J,u}$ defined in Section~\ref{s: anti-conc}.
Now, let $\Event_{M,J,u}'\subset \Event_{M,J,u}$ be the event that
the uniform random normal $\Proj_{E_{i}^\perp}(G)$
satisfies the following condition:
$$
 \exists \widetilde J \subset [n] \quad \mbox{ with } \quad
  \vert \widetilde J\vert \leq 2\Big(\frac{n-i}{n}\Big)^{\gamma/2}n
  \quad \mbox{ such that }
 $$
\begin{align*}
 \forall \lambda\in\C:\,\Big|\Big\{j\in n\setminus \widetilde J:\, |\langle \Proj_{E_{i}^\perp}(G),e_j\rangle -\lambda|\leq
\exp\Big(- C_0
\Big(\frac{n}{n-i}\Big)^\gamma\Big)\Big\}\Big|\leq n-i,
\end{align*}
where
$C_0$ is the constant from Theorem~\ref{th: random normal structure}.
Note that conditioned on the event $\Event_{M,J,u}$ the subspace $E_{i}$ is completely determined by
$M$ and $J$, in particular it is fixed within the event $\Event_{M,J,u}$.
Therefore, by the independence of $G$ from the triple $(A_n,A_n',\sigma_n)$, we have that $\Proj_{E_{i}^\perp}(G)$ and the $u$-th row of $A_n'$
are independent conditioned on $\Event_{M,J,u}$.
Then, conditioning on the event $\Event_{M,J,u}'$ and denoting
$$
 B_z':=d^{-1/2}A_n'-z\idmat
$$
we apply Proposition~\ref{prop: anti-c}
with $y=\Proj_{E_{i}^\perp}(G)$ and $\lambda=d^{1/2}\langle y,\row_u(z\idmat)\rangle$,
which gives that
\begin{align}
\Prob\Big\{&|\langle \Proj_{E_{i}^\perp}(G),\row_u(B_z')\rangle|\leq
(16d)^{-1/2}\exp\Big(- C_0
\Big(\frac{n}{n-i}\Big)^\gamma\Big)\, \, \big|
\, \, \Event_{M,J,u}'\Big\}\nonumber\\
&\leq 144\,  \frac{n-i}{n} + \Big(16\, \Big(\frac{n-i}{n}\Big)^{\gamma/2} \Big)^d+n^{-1/10}\leq 145\, \frac{n-i}{n},
\label{eq: aux pquhqp3}
\end{align}
provided that $d$ is large enough.
For convenience, we denote $q:=i-\lfloor n^{1/4}\rfloor$.
Define another (the last) auxiliary event
\begin{align*}
\widetilde\Event_{M,J,u}:=
\Event_{M,J,u}\cap\Big\{&\ln (n/(n-i))\,\big\|\Proj_{E_{q}^\perp}(\row_u(B_z'))\big\|_2
\geq |\langle \row_u(B_z'),\Proj_{E_i^\perp}(G)\rangle|\Big\}.
\end{align*}
Using  the deterministic relation
$$\big\|\Proj_{E_i^\perp}(\row_u(B_z'))\big\|_2
\leq \big\|\Proj_{E_{q}^\perp}(\row_u(B_z'))\big\|_2,$$
the independence of $\row_u(A_n')$ and $\Proj_{E_{i}^\perp}(G)$ conditioned on $\Event_{M,J,u}$
and (\ref{l: line projection}) applied with $t=\ln (n/(n-i))$, we obtain
$$\Prob(\widetilde\Event_{M,J,u}\,|\,\Event_{M,J,u})\geq 1-\frac{n-i}{n},$$
and thus
$$
\Prob(\widetilde\Event_{M,J,u}^c\,|\,\Event_{M,J,u}') \leq \frac{\Prob(\Event_{M,J,u}^c\cap \Event_{M,J,u})}{\Prob(\Event_{M,J,u}')}\leq \frac{n-i}{n}\cdot\frac{\Prob(\Event_{M,J,u})}{\Prob(\Event_{M,J,u}')}.
$$
Together with \eqref{eq: aux pquhqp3} and using that
$$
  4\sqrt{d}\, \ln(n/(n-i))\leq \exp\Big( \Big(\frac{n}{n-i}\Big)^\gamma\Big)
$$
for sufficiently large $d$, we  get for an appropriate choice of the constantn $\widetilde C$ that
\begin{align*}
\Prob\Big\{&\big\|\Proj_{E_{q}^\perp}(\row_{\sigma_n(i)}(B_z'))\big\|_2\leq
\exp\Big(-\widetilde C\Big(\frac{n}{n-i}\Big)^\gamma\Big)\, \, \big|\, \, \Event_{M,J,u}'\Big\}\\
&
\leq \Prob\Big\{|\langle \Proj_{E_{i}^\perp}(G),\row_u(B_z')\rangle|\leq
c\,d^{-1/2}\exp\Big(- C_0
\Big(\frac{n}{n-i}\Big)^\gamma\Big)\, \, \big|\, \, \Event_{M,J,u}'\Big\}
+\Prob(\widetilde\Event_{M,J,u}^c\, \, \big|\, \, \Event_{M,J,u}')\\
&\leq \frac{n-i}{n}\left(145 + \frac{\Prob(\Event_{M,J,u})}{\Prob(\Event_{M,J,u}')}\right) \leq
146 \,\, \frac{n-i}{n}\,\,  \frac{\Prob(\Event_{M,J,u})}{\Prob(\Event_{M,J,u}')} .
\end{align*}
Using the independence $G$ and $(A_n, A_n', \sigma_n)$ and
applying Theorem~\ref{th: random normal structure} with
$I= E_{i}$, which is fixed within the event $\Event_{M,J,u}$,
we observe
$$
  \Prob\Big(\bigcup_{M,J,u}\Event_{M,J,u}'\Big)\geq 1- \frac{n-i}{n}.
$$
Note also that the events $\Event_{M,J,u}$ are pairwise disjoint, so that $\sum_{M,J,u}\Prob(\Event_{M,J,u})\leq 1$.
Therefore,
using  that
$\Event_{M,J,u}'\subset \Event_{M,J,u}$
we obtain
\begin{align*}
\Prob\Big\{&\big\|\Proj_{E_{q}^\perp}(\row_{\sigma_n(i)}(B_z'))\big\|_2\leq
\exp\Big(-\widetilde C\Big(\frac{n}{n-i}\Big)^\gamma\Big)\Big\}\\
&\leq \sum_{M,J,u}\Prob\Big\{\big\|\Proj_{E_{q}^\perp}(\row_{\sigma_n(i)}(B_z'))\big\|_2\leq
\exp\Big(-\widetilde C\Big(\frac{n}{n-i}\Big)^\gamma\Big)\, \, \big|\, \, \Event_{M,J,u}'\Big\}\, \Prob\Big\{\Event_{M,J,u}'\Big\}
\\ &+\Prob\Big(\Big[ \bigcup_{M,J,u}\Event_{M,J,u}'\Big]^c\Big)
\leq \Big(146\, \frac{n-i}{n} \Big)\, \sum_{M,J,u}\Prob(\Event_{M,J,u}) +\frac{n-i}{ n}\leq 147\, \frac{n-i}{ n}.
\end{align*}
Note that for any realization $(M, M', \sigma)$ of $(A_n, A_n', \sigma_n)$ we have
$\row_{\sigma(j)}(M')=\row_{\sigma(j)}(M)$ for all $j<q$, therefore
$$
 E_q=\spn \{\row_{\sigma_n(j)}(B_z)\}_{j<q}=\spn \{\row_{\sigma_n(j)}(B_z')\}_{j<q},
$$
Thus
\begin{align*}
\Prob\Big\{&\d\big(\row_{\sigma_n(i)}(B_z'), \, \spn_{j<q}\{\row_{\sigma_n(j)}(B_z')\})
\leq
\exp\Big(-\widetilde C\Big(\frac{n}{n-i}\Big)^\gamma\Big)\Big\}\leq 147\, \frac{n-i}{n}.
\end{align*}
In view of the independence of $\sigma_n$ and $A_n'$, we can replace the row
$\row_{\sigma_n(i)}(B_z')$ in the above formula with $\row_{\sigma_n(q)}(B_z')$
with no change to the probability estimates.
Since $A_n'$ and $A_n$ are equidistributed we can also replace $\row_{\sigma_n(q)}(B_z')$ and $\row_{\sigma_n(j)}(B_z')$
with $\row_{\sigma_n(q)}(B_z)$ and $\row_{\sigma_n(j)}(B_z)$. Finally note that in our range of $i$,
$\frac{n-i}{n}$ is equivalent to $\frac{n-q}{n}$ up to constant 2 and  that $n-n/d^3\leq q\leq n-n/\ln^{1/\gamma} n -n^{1/4}$.
This completes the proof of Lemma~\ref{l: step I}.
\qed

\subsection*{Acknowledgments}

P.Y.\ was supported by grant ANR-16-CE40-0024-01.
A significant part of this work was completed while the last three named authors were in residence at the
Mathematical Sciences Research Institute in Berkeley, California,
supported by NSF grant DMS-1440140, and the first two named authors visited the institute. The hospitality of MSRI and of the
organizers of the program on Geometric Functional Analysis and
Applications is gratefully acknowledged.

\address

\end{document}